\documentclass[11pt, english, reqno]{amsart}
\usepackage{fullpage}
\usepackage[utf8]{inputenc}
 \usepackage{amsmath,amssymb,stmaryrd,array,multirow,dsfont,mathrsfs}
\usepackage{amsthm}
\usepackage{enumerate}
\usepackage{xcolor}
\usepackage{color}
\usepackage{hyperref}
\usepackage{indentfirst}

\usepackage[utf8]{inputenc}
\usepackage[T1]{fontenc}
\usepackage{microtype}

\usepackage{graphics}
 \usepackage{amsmath,amssymb,stmaryrd,array,multirow,dsfont,mathrsfs}
\usepackage{amsthm}
\usepackage{xcolor}
\usepackage{color}
\usepackage{hyperref}

\graphicspath{{Images/}}

\xdefinecolor{darkgreen}{rgb}{0,0.4,0}

\newcommand{\beq}{\begin{equation}}
\newcommand{\eeq}{\end{equation}}
\newcommand{\beqs}{\begin{equation*}}
\newcommand{\eeqs}{\end{equation*}}
\newcommand{\ben}{\begin{eqnarray}}
\newcommand{\een}{\end{eqnarray}}
\newcommand{\beno}{\begin{eqnarray*}}
\newcommand{\eeno}{\end{eqnarray*}}
\renewcommand{\Re}{{\rm Re}\,}
\renewcommand{\Im}{{\rm Im}\,}
\DeclareMathOperator{\Ran}{Ran}
\DeclareMathOperator{\Span}{Span}

\DeclareMathOperator{\Div}{div}

\newcommand{\Rmnum}[1]{\uppercase\expandafter{\romannumeral #1} }
 \numberwithin{equation}{section}
\usepackage{mathtools}

\DeclarePairedDelimiterX{\inp}[2]{\langle}{\rangle}{#1, #2}

\newtheorem{thm}{Theorem}[section]
\newtheorem{lem}[thm]{Lemma}
\newtheorem{prop}[thm]{Proposition}
\newtheorem{rmk}[thm]{Remark}

\def \d {\mathrm {d}}

\def\cB{{\mathcal B}}
\def\cC{{\mathcal C}}

\def\cF{{\mathcal F}}

\def\cM{{\mathcal M}}
\def\cN{{\mathcal N}}
\def\cO{{\mathcal O}}

\def\cU{{\mathcal U}}

\let\f=\frac
\def \p {\partial}
\def\mC {\mathbb{C}}
\def\mR {\mathbb{R}}
\def\mN {\mathbb{N}}
\def\mZ {\mathbb{Z}}

\def \pt {\partial_{t}}

\def \diag {\textnormal{diag}}

\def \uu {\underline{u}}
 \def \ue {\underline{E}}

\def \uU {\underline{U}}

\title{Spectral instability of small-amplitude periodic waves\\of the electronic Euler-Poisson system}

\author{Pascal Noble}
\address{Universit\'e de Toulouse, CNRS, IMT - UMR 5219, INSA, F-31077 Toulouse, France}
\email{pascal.noble@math.univ-toulouse.fr}

\author{Luis Miguel Rodrigues}
\address{Univ Rennes \& IUF, CNRS, IRMAR - UMR 6625, F-35000 Rennes, France}
\email{luis-miguel.rodrigues@univ-rennes1.fr}

\author{Changzhen Sun*}
\address{Universit\'e de Toulouse, CNRS, IMT - UMR 5219, UPS, F-31062 Toulouse Cedex 9, France}
\email{changzhen.sun@math.univ-toulouse.fr}
\thanks{* Corresponding author}  %Work of  C.S. is supported by the ANR LabEx CIMI (grant ANR-11-LABX-0040) within the French State Programme ``Investissement d'Avenir''}

\date{\today}

\begin{document}

\maketitle

\begin{abstract}
\noindent The present work shows that essentially all small-amplitude periodic traveling waves of the electronic Euler-Poisson system are spectrally unstable. This instability is neither modulational nor co-periodic, and thus requires an unusual spectral analysis and, beyond specific computations, newly devised arguments. The growth rate with respect to the amplitude of the background waves is also provided when 
the instability occurs. 

\vspace{2em}

%\begin{center}

\vspace{2em}

%\begin{center}

\end{abstract}

\section{Introduction}

The present work achieves a two-fold goal. On one hand, we are interested in understanding plasma dynamics, in its own right. On the other hand, we intend to contribute to the young but rapidly growing field studying spectral instabilities of periodic traveling waves that are neither modulational nor co-periodic.

\subsection{The Euler-Poisson system}

We consider the electronic Euler-Poisson system without magnetic field, a hydrodynamical model of plasma which describes the dynamics of electrons, coupled to a background density of ions through a self-consistent electric field. In such a model, one neglects the ions motion, to account for the time-scale separation induced by the small mass ratio between electrons and ions.

The dynamics of electrons is then described by the following Euler-Poisson system
\begin{equation}\label{ep3d}
\left\{
\begin{array}{rl}
\displaystyle
\partial_t \rho+\Div(\rho{\bf u})&=0\,,\vspace{2mm}\\
\displaystyle
\partial_t(\rho{\bf u})+\Div(\rho{\bf u}\otimes{\bf u})+\nabla P(\rho)&=\rho\nabla\phi\,,\vspace{2mm}\\
\displaystyle
\epsilon\,\Delta\phi&=\rho-\rho_{i}\,.
\end{array}
\right.
%\forall t>0,\forall x\in\mathbb{R}^3.
\end{equation}
Here the unknowns $\rho(t,{\bf x})\in\mathbb{R}^+$, ${\bf u}(t,{\bf x})\in\mathbb{R}^3$, $\nabla\phi(t,{\bf x})\in \mathbb{R}^3$ are the electron density, the electron velocity and the self-consistent electric field at time $t\in\mathbb{R}$ and position ${\bf x}\in\mathbb{R}^3$. The reference ion density $\rho_{i}$ is fixed and could depend on ${\bf x}$ but not on $t$. In the present contribution, we assume $\rho_{i}$ to be constant so that the system remains invariant by spatial translations and it makes sense to consider traveling-wave solutions. The thermal pressure of electrons $P(\rho)$ is often assumed to follow a polytropic $\gamma$ law, $P(\rho)=T\rho^\gamma$ with $T\in\mathbb{R}^+$ and  $\gamma\geq 1$. The constant $\epsilon$ is the square of the quasi-neutral parameter, a multiple of the Debye length. The associated dynamics preserves curl-free constraints on velocity fields and, when restricted to irrotational velocities, System~\eqref{ep3d} admits a Hamiltonian formulation
\[
\p_t\begin{pmatrix}\rho\\{\bf u}\end{pmatrix}
\,=\,\mathcal{J}\delta_{(\rho,{\bf u})}{\bf H}[\rho,{\bf u}]
\]
with skew-adjoint operator $\mathcal{J}$ and Hamiltonian density ${\bf H}$ given by
\begin{align*}
\mathcal{J}\,&:=\,\begin{pmatrix}0&-\Div\\-\nabla&0\end{pmatrix}\,,&
{\bf H}[\rho,{\bf u}]
&:=\frac12\rho\,\|{\bf u}\|^2+F(\rho)+\epsilon\,\|\nabla\phi\|^2\,,
\end{align*}
where $\delta$ denotes the variational derivative, the internal energy $F$ satisfies $F''(\rho)=P'(\rho)/\rho$ and $\phi=\epsilon^{-1}\,\Delta^{-1}(\rho-\rho_i)$ is thought as a function of $\rho$.

Note that when, as here, $\rho_{i}$ is taken to be constant, a steady constant solution is obtained by setting $\rho\equiv\rho_i$ and taking ${\bf u}$ constant and $\phi$ zero. Local well-posedness in suitable Sobolev spaces near such constant states requires $P'(\rho_i)$ to be nonnegative and we assume its positivity throughout. There is a large body of literature concerning the stability of the foregoing constant states. We refer the reader to \cite{CMP-Guo,SIMA-Germain-Masmoudi-Pausader,IMRN-Ionescu-Pausader,FM-Jang-Li-Zhang,JEMS-Li-Wu,ARMA-Guo-Han-Zhang, CMP-FZ} for both significant contributions and relevant introductions to the extensive literature. Let us stress that such steady solutions are at best marginally spectrally stable so that the above references rely on dispersive estimates and normal forms, thus they are relatively sensitive to the choice of the pressure law, especially the value  $(\f{\epsilon}{\rho}\p_{\rho}P)(\rho_i).$

In contrast, other equilibria or relative equilibria --- such as traveling waves --- of the electronic Euler-Poisson system, even those of small amplitude, seem to have received no attention at all. The main reason is surely that none of the traveling waves whose existence requires the co-existence of two distinct equilibria at the same phase speed may exist for this system. Among plane waves this excludes both solitary\footnote{The reduced traveling wave profile ODE is planar so that the existence of a homoclinic loop implies the existence of another constant state.} waves and kinks. Incidentally we point out that the situation is already significantly better for the two-species or ionic versions of the system; see for instance \cite{Degond-et-al,PhysD-Haragus-Scheel,JMFM-Haragus-Scheel,JDE-Bae-Kwon,ARMA-Bae-Kwon}.

In the present work, we consider plane periodic traveling waves, that is, solutions to \eqref{ep3d} of the form
\[
(t,{\bf x})\mapsto (\underline{\rho},\underline{{\bf u}})\left({\bf k}\cdot\left({\bf x}-t\,{\bf V}\right)\right)
\]
with wave vector ${\bf k}\in\mathbb{R}^3$, wave speed ${\bf V}\in\mathbb{R}^3$ and $1$-periodic profile $(\underline{\rho},\underline{{\bf u}})$ depending on a one-dimensional variable. Note that for profiles that are non characteristic in the weak sense that the set where $\underline{\rho}\,(\underline{{\bf u}}-\underline{{\bf V}})$ vanishes has empty interior, an elementary computation shows that the component of $\underline{{\bf u}}$ orthogonal to ${\bf k}$ is constant. Therefore one may use Galilean invariance to reduce to the case when $\underline{{\bf u}}$ and ${\bf k}$ are co-linear. A suitable nondimensionalization may also bring the reduction $\epsilon=1$, $\rho_i=1$. Hereafter we shall make all these notational simplifications.

Since we are using it a few times, let us make explicit that here the Galilean invariance is the fact that if $(t,{\bf x})\mapsto (\rho,{\bf u})(t,{\bf x})$ solves \eqref{ep3d} then so does $(t,{\bf x})\mapsto (\rho,{\bf u}+{\bf a})(t,{\bf x}-t\,{\bf a})$, for any fixed ${\bf a}\in\mathbb{R}^3$.

The upshot of our main achievement is that essentially any plane periodic traveling wave of sufficiently small amplitude is spectrally unstable under localized perturbations. To achieve this goal, by a well-known combination of elementary Fourier and spectral arguments, it is sufficient (but not necessary) to prove spectral instability under perturbations with the same planar symmetry as the background wave. We refer the reader to \cite{SMF-Audiard-Rodrigues} for a detailed version of the argument, specialized to waves of the nonlinear Schr\"odinger equations. From now on, we thus specialize to one-dimensional solutions, that is solutions that for some fixed unitary vector ${\bf e}$, depend only on $(t,{\bf e}\cdot {\bf x})$ and whose velocity field points everywhere in the direction of ${\bf e}$. Then, denoting $x:={\bf e}\cdot {\bf x}$ and $u={\bf e}\cdot {\bf u}$, System~\eqref{ep3d} becomes
\begin{equation}\label{ep1d}
\left\{
\begin{array}{rl}
\displaystyle
\partial_t \rho+\p_x(\rho\,u)&=0\,,\vspace{2mm}\\
\displaystyle
\partial_t(\rho\,u)+\p_x(\rho\,u^2)+\p_x P(\rho)&=\rho\p_x\phi\,,\vspace{2mm}\\
\displaystyle
\p_x^2\phi&=\rho-1\,.
\end{array}
\right.
\end{equation}
Obviously one-dimensional solutions are curl-free and, consistently, System~\eqref{ep1d} inherits a one-dimensional version of the Hamiltonian formulation given hereabove.

There is a further reduction that we want to perform. The goal of this step is two-fold: to obtain another Hamiltonian formulation where the skew-adjoint operator $\mathcal{J}$ is replaced with an invertible one and to get rid of the Poisson equation. With this in mind, let us first observe that if
\[
(t,x)\mapsto (\underline{\rho},\underline{u})\left(k\left(x-t\,V\right)\right)
\]
defines a traveling-wave solution to \eqref{ep1d} with wave number $k\in\mathbb{R}$, wave speed $V\in\mathbb{R}$ and $1$-periodic profile $(\underline{\rho},\uu)$, then the relative discharge rate $\underline{\rho}(\uu-V)$ is constant, thus so is $\uu+(\uu-V)\,\p_x^2\underline{\phi}$. By a further Galilean transformation one may force $\uu+(\uu-V)\,\p_x^2\underline{\phi}\equiv 0$ and we do so henceforth. Now, let us observe that more generally the first equation of System~\eqref{ep1d} implies that $\p_t\p_x\phi+(1+\p_x^2\phi)u$ is constant with respect to $x$, so that by restricting to localized perturbations of such a background wave, one derives $\p_t\p_x\phi+(1+\p_x^2\phi)u=0$. As a result for the class of solutions we are interested in it is sufficient to set\footnote{Note that $E$ is the opposite of the electric field.} $E=\p_x\phi$ and replace \eqref{ep1d} with
\beq\label{eep-int}
\left\{
\begin{array}{l}
\displaystyle
\pt E+(1+\p_x E)u=0\,,\vspace{2mm}\\
 \displaystyle
  \pt u+u \p_x u+\p_x(F'(1+\p_x E))= E\,.
  \end{array}
  \right.
\eeq
It is even more straightforward to check that solutions to \eqref{eep-int} do yield solutions to \eqref{ep1d}.

Therefore one may restrict the analysis to the study of System~\eqref{eep-int}. The latter also admits a Hamiltonian formulation,
\[
\p_t\begin{pmatrix}E\\u\end{pmatrix}
\,=\,J\delta_{(E,u)}\mathcal{H}[E,u]
\]
with skew-adjoint operator $J$ and Hamiltonian density $\mathcal{H}$ given by
\begin{align*}
J\,&:=\,\begin{pmatrix}0&-1\\1&0\end{pmatrix}\,,&
\mathcal{H}[E,u]
&:=\frac12(1+\p_x E)\,u^2+F(1+\p_x E)+\frac12\,E^2\,.
\end{align*}
Let us observe that for System~\eqref{eep-int} the only constant solution is $(E,u)\equiv(0,0)$. We have indeed lost the freedom to fix the velocity field to an arbitrary constant in normalizations leading to System~\eqref{eep-int} and based on Galilean invariance. Moreover if
\[
(t,x)\mapsto (\ue,\uu)\left(k\left(x-t\,V\right)\right)
\]
defines a traveling-wave solution to \eqref{eep-int} with $1$-periodic profile $(\ue,\uu)$, then for some constant $\mu$
\begin{equation}\label{profile-intro}
 \frac12 \ue^2+ \mathcal{W}(1+k\,\ue';V) \equiv \mu
\end{equation}
where
\begin{align*}
\p_\rho\mathcal{W}(\rho;V)&:=(\rho-1)\,h(\rho;V)\,,&
h(\rho;V)&:=\frac{V^2}{\rho^3}-F''(\rho)=\frac{V^2-P'(\rho)\,\rho^2}{\rho^3}\,.
\end{align*}
Since $\p_\rho^2\mathcal{W}(1;V)=h(1;V)=V^2-P'(1)$, a family of small amplitude periodic waves with velocity $V$ may exists only if $V^2\geq P'(1)$ and, as we prove later, the supersonic condition $V^2>P'(1)$ is sufficient to guarantee the existence of such a family. Besides, in such cases, normalizing with $\mathcal{W}(1;V)\equiv 0$ ensures that the small-limit limit is equivalently described by sending $\mu$ to zero from above.

Before stating our main result, we point out that since when $(t,x)\mapsto (E,u)(t,x)$ solves \eqref{eep-int} so does $(t,x)\mapsto (-E,-u)(t,-x)$, it is sufficient to consider cases when $V\geq0$.

\begin{thm}\label{main}
Assume that $P$ is smooth in a neighborhood of $1$ with $P'(1)>0$. There exist smooth functions $\delta_0:(\sqrt{P'(1)},+\infty)\to \mathbb{R}_+^*$, $\Gamma:(\sqrt{P'(1)},+\infty)\to \mathbb{R}$ and $\delta_1:\Gamma^{-1}(\mathbb{R}^*)\to \mathbb{R}_+^*$ $\omega_1:\Gamma^{-1}(\mathbb{R}^*)\to \mathbb{R}_+^*$ such that on
\[
\Omega:=\left\{\,(\delta,V)\,;\,V>\sqrt{P'(1)}\,,\,0\leq \delta<\delta_0(V)\,\,\right\}
\]
there exists smooth functions
\begin{align*}
\Omega&\to \mathbb{R}_+^*\times \mathcal{C}^\infty(\mathbb{R}/\mathbb{Z})^2\,,&
(\delta,V)\mapsto (k^{(\delta,V)},\ue^{(\delta,V)}(\cdot),\uu^{(\delta,V)}(\cdot))
\end{align*}
satisfying the following.
\begin{enumerate}
\item For any $V>\sqrt{P'(1)}$, $(\ue^{(0,V)}(\cdot),\uu^{(0,V)}(\cdot))\equiv (0,0)$.
\item For any $(\delta,V)\in\Omega$, $(k^{(\delta,V)},\ue^{(\delta,V)},\uu^{(\delta,V)})$ defines a periodic traveling wave to \eqref{eep-int} with speed $V$, solving \eqref{profile-intro} with $\mu=h(1;V)\,\delta^2/2$.
\item For any $V>\sqrt{P'(1)}$ such that $\Gamma(V)\neq0$ and $0<\delta<\delta_1(V)$, the corresponding traveling wave is spectrally unstable to localized perturbations with growth rate at least $\omega_1(V)\,\delta^3$.
\end{enumerate}
Furthermore, if $P$ is analytic in a neighborhood of $1$, then $\Gamma$ is also analytic so that $\Gamma^{-1}(\{0\})$ is either discrete or equal to $(\sqrt{P'(1)},+\infty)$.
\end{thm}

In the foregoing statement, $\delta$ plays the role of a small-amplitude parameter. The use of $\delta$ instead of $\mu$ is motivated by the fact that then one may enforce
\begin{align*}
k^{(\delta,V)}\,(\ue^{(\delta,V)})'&\stackrel{\delta\to0}{=}\delta\,\cos(2\pi\,\cdot)+\mathcal{O}(\delta^2)\,,&
\uu^{(\delta,V)}&\stackrel{\delta\to0}{=}\,V\,\delta\,\cos(2\pi\,\cdot)+\mathcal{O}(\delta^2)\,,&
\end{align*}

The function $\Gamma$ is an instability index whose non vanishing ensures that corresponding small-amplitude waves are unstable. We provide a relatively explicit --- but quite awful --- formula for $\Gamma$ in Proposition~\ref{prop-rmc}. The vanishing of $\Gamma$ should be thought as rare and we indeed prove in Proposition~\ref{th5} that in the case $P(\rho):=T\,\rho^\gamma$, with $T>0$, $\gamma\geq 1$, the index $\Gamma$ vanishes at most a finite number of times.

Moreover, $\Gamma$ detects only the strongest kind of instabilities due to one-dimensional perturbations. However, as we sketch below, there is an infinite number of similar indices tracking possibly weaker one-dimensional instabilities and whose vanishing is expected to be somehow independent of the one of $\Gamma$. Thus we believe that all small-amplitude waves are unstable.

\subsection{Periodic waves of Hamiltonian systems}

We believe that the present case study also provides significant insights for the stability analysis of periodic waves of Hamiltonian systems taken in a broad sense and we take some time now to place it within this perspective. When doing so, we allow ourselves to use standard technical terminology of the field whose definition is explicitly recalled later in the text. With this in mind, we refer the reader to \cite{KapitulaPromislow_book} for general background on nonlinear wave dynamics, to \cite{Angulo-Pava,PhysD-Haragus-Kapitula,dBRN} for material more specific to Hamiltonian systems and to \cite{R,R_Roscoff} for some perspectives on periodic traveling waves.

Let us first recall that, as of now, there isn't in the literature even a single instance of proof of nonlinear stability under localized perturbations of a periodic traveling wave of a Hamiltonian system. This is in strong contrast with on one side the situation for dissipative systems, in particular parabolic ones, and on the other side the situation for periodic perturbations of the same period. We refer the reader to \cite{ARMA1-JNRZ,Inventiones-JNRZ} on the former. As for the latter, the starting point is that with periodic boundary conditions one may integrate conservation laws accompanying the system and use a Lyapunov-type argument to obtain bounded orbital stability. The precise analysis may be thought as an extension of the classical Grillakis-Shatah-Strauss theory. On the latter we refer the reader to \cite{Angulo-Pava,dBGRN,dBRN,Nonlinearity-BMR}. Likewise, to our knowledge there is only one single contribution \cite{JFA-R} examining linear stability under localized perturbations, the latter being the only piece of work discussed here analyzing dispersive properties of the dynamics near periodic waves. At the nonlinear level, for localized perturbations, it is worth mentioning that some results converting spectral instability into nonlinear instability do exist; see \cite{ARMA-JLL}. Let us however warn the reader that the instability result proved there is of orbital type, and not of space-modulated type, which is the notion of stability proved to be sharp for parabolic systems in \cite{Inventiones-JNRZ}; for the sake of comparison, see \cite{DR2} for examples of proofs of nonlinear space-modulated instabilities (for some hyperbolic equations).

In turn, spectral studies are well-developed and we provide here only a brief overview of the literature. To begin with, we point out that results proving spectral stability for localized perturbation of waves of arbitrary size are quite rare. Most of them have been obtained for systems that are completely integrable by inverse scattering through a Lax pair representation. We refer the reader to \cite{Upsal-Deconinck} for a quite systematic exposition of the argument and to references therein for a large set of applications. A notable exception, using arguments that are arguably less systematic and more special\footnote{In the same sense as the study of scalar reaction-diffusion equations through Sturm-Liouville theory.} to the systems at hand, is the remarkable stability classification of waves of nonlinear Klein-Gordon equations \cite{JDE-JMMP}.

Concerning spectral instabilities of waves of arbitrary size, there are essentially two kinds of results available. On one hand, it is possible a count modulo two of positive eigenvalues associated with perturbations of the same period. The corresponding indices may be obtained either by a Krein signature count \cite{BJK} or by an Evans function computation \cite{Nonlinearity-BMR}. Note that the latter point of view has the advantage of being quite robust and has been indeed applied previously to many other kinds of waves and systems. For a quite large class of Hamiltonian systems, the corresponding index has been computed in both small-amplitude and large-period regimes \cite{IUMJ-BMR}. On the other hand, one may analyze possible slow side-band instabilities, that is, those corresponding to spectrum near the origin that are almost co-periodic. This relies on the understanding of the structure of nearby traveling waves and may be carried out either by direct Floquet-Bloch expansions as described in \cite{JNS-BNR,BHJ} or through Evans functions computations as in \cite{SMF-Audiard-Rodrigues}. One of the points of the latter is that it may be connected to formal geometrical optics and modulation theory. This link is extremely robust and has also been proved for parabolic equations in \cite{Serre,Noble-Rodrigues}, lattice dynamical systems in \cite{KR}, and discontinuous waves of some hyperbolic systems in \cite{JNRYZ}. For a quite large class of Hamiltonian systems, the corresponding slow modulation criterion has been elucidated in both small-amplitude and large-period regimes \cite{Nonlinearity-BMR2}. It is worth pointing out that instabilities due to the non trivial part of the small-amplitude slow modulation criterion are often designated as Benjamin--Feir instabilities, in reference to  its formal derivation for Stokes waves of free-surface water motion, and that they have been the object of many direct studies. See, in particular, \cite{Bridges-Mielke,Inventiones-BMV} for rigorous proofs of the original Benjamin--Feir instability.

We now specialize the discussion to spectral stability/instability of small-amplitude periodic waves. Beyond the results already mentioned, small-amplitude spectral stability results that have been obtained so far \cite{JMFM-HLS,JDE-Gallay-Haragus,SIMA-Johnson} rely on a two-tier argument. On one hand one may rule out slow modulational instabilities by a variation on the arguments mentioned above. On the other hand in cases considered there one may prove that the latter was the only possible instability. This follows by the quite robust argument that instability may only arise from the collision of eigenvalues with opposite Krein signatures; see \cite{McKay,KKS,SIAM-KM}.

The situation for small-amplitude waves of the electronic Euler--Poisson system is dramatically different in the sense that at zero amplitude there are not a single multiple eigenvalue but infinitely many double eigenvalues with opposite signatures, including the one corresponding to slow modulational perturbations. In this sense it is similar to the one occurring for Stokes waves and indeed, to our knowledge, the only one other contribution aiming at the fully rigorously analysis of such a situation is devoted to Stokes waves ; see the recent %\footnote{Dating from March 2022 in its corrected form.} preprint
work \cite{Hur-Yang}. Combining this new non-modulational analysis with classical Benjamin--Feir instability they prove %the remarkable fact 
that all small-amplitude Stokes waves are spectrally unstable. Let us stress that whereas we perform a direct Floquet-Bloch analysis, the authors of \cite{Hur-Yang} develop an Evans function approach,
%A drawback of the Evans function approach is that it 
which does not seem to provide easily information about expected growth rates. %\textcolor{red}{
For a comparison of traditional advantages of the two kinds of approaches we refer to \cite[p.47]{R}.

From a technical point of view, our mathematical analysis is closer in spirit to the formal asymptotic analyses in \cite{SIADS-TDK,SIADS-CDT,JFM-CDT}. The latter differ from fully mathematical proofs only by the fact that they \emph{assume} the existence of smooth expansions for all quantities of interest, a highly non trivial fact when multiple eingenvalues are to be perturbed. We stress that indeed, already in the by now classical slow modulation theory, the main step of rigorous analyses is to prove that this smoothness holds despite the initial presence of Jordan blocks, a result that is highly non generic from the point of view of general linear algebra.

\subsection{General mechanism}

We believe that the mechanism yielding the spectral instability proved here is quite robust. Thus we outline in the introduction the main wheels of the machinery. The strategy of the actual proof is precisely to prove the claims sketched below and compute associated key coefficients.

To begin with, for the sake of readability, let us drop any mark of the dependence on $V$ in the present discussion and denote $L^\delta$ the operator with $1$-periodic coefficients obtained by linearizing \eqref{eep-int} about a traveling wave of parameter $(\delta,V)$, in suitably scaled co-moving coordinates. Since we are interested in localized perturbations, the operator is defined as acting on functions living in a $L^2(\mathbb{R})$-based space. Then, exactly as Fourier transform is used for constant-coefficient operators, one may use the Floquet-Bloch transform to decompose the action of $L^\delta$ as the action of multipliers $L^\delta_\xi$ by Floquet multiplier $\xi$. Each multiplier $L^\delta_\xi$ is a differential operator acting on functions living in a $L^2(\mathbb{R}/\mathbb{Z})$-based space and has compact resolvents hence discrete spectra consisting entirely of eigenvalues of finite multiplicity depending continuously on $\xi$. Moreover, as a consequence, the spectrum of $L^\delta$ is the union over $\xi$ of the spectra of all $L^\delta_\xi$.

A key element of the analysis is that $L^\delta$ inherits form the Hamiltonian structure the factorization $L^\delta\,=\,J\,A^\delta$, with $A^\delta$ a self-adjoint operator. This yields the Hamiltonian symmetry $(L^\delta)^*=-J^{-1}\,L^\delta\,J$, and correspondingly $(L_\xi^\delta)^*=-J\,L_\xi^\delta\,J^{-1}$. This readily implies the classical Hamiltonian properties that spectral stability is possible only if the spectrum lies on the imaginary axis and that only multiple eigenvalues may leave the imaginary axis when varying parameters.

Note that recasting Fourier analysis in terms of Bloch transform provides a complete description of the spectrum of $L^0_\xi$ for any $\xi$, and that eigenvalues are at most of multiplicity $2$. Let $\lambda_0\in i\,\mathbb{R}$ be a double eigenvalue of $L^0_{\xi_0}$ for some $\xi_0$. Then when $(\delta,\xi-\xi_0)$ is sufficiently small, the spectrum of $L^\delta_\xi$ near $\lambda_0$ coincides with the one of any two-by-two matrix
\[
D_{\xi}^{\delta}\,=\,\left(\inp{\tilde{q}_\ell^{\delta}(\xi,\cdot)}{L_{\xi}^{\delta}q_m^{\delta}(\xi,\cdot)}_{L_{per}^2}\right)_{1\leq \ell, m\leq 2}
\]
defined from $(q_1^{\delta}(\xi,\cdot),q_2^{\delta}(\xi,\cdot))$ a basis of the sum of characteristic spaces of $L_\xi^\delta$ associated with eigenvalues near $\lambda_0$, and $(\tilde{q}_1^{\delta}(\xi,\cdot),\tilde{q}_2^{\delta}(\xi,\cdot))$ a dual basis of the corresponding space of $(L_\xi^\delta)^*$ associated with eigenvalues near $\overline{\lambda_0}$. In the foregoing, $\inp{\cdot}{\cdot}_{L_{per}^2}$ denotes the canonical $L^2$ scalar product on $L^2(\mathbb{R}/\mathbb{Z})$, skew-linear in its first variable. The classical Kato perturbation theory allows to construct all the quantities involved hereabove smoothly and even to prescribe a choice when $\delta=0$ and to extend it while preserving real and Hamiltonian symmetries.

Now, at $\delta=0$, Fourier analysis and Hamiltonian symmetry provides bases  $(p_1^{0}(\xi,\cdot),p_2^{0}(\xi,\cdot))$ and $(\tilde{p}_1^{0}(\xi,\cdot),\tilde{p}_2^{0}(\xi,\cdot))$ for the direct and dual spaces, of the form $p_m^{0}(\xi,\cdot)=e^{2\pi\,i\,j_1\,\cdot}\,P_m(\xi)$, $m=1,2$ and $\tilde{p}_\ell^{0}(\xi,\cdot)\,=\,i\,J^{-1}{p}_\ell^{0}(\xi,\cdot)$, $\ell=1,2$, where $(j_1,j_2)$ are some fixed integers independent of $\xi$ and distinct. We warn the reader that at this stage the two bases are not in duality. However, on one hand, since $j_1\neq j_2$, $p_m^{0}(\xi,\cdot)$ and $\tilde{p}_\ell^{0}(\xi,\cdot)$ are already orthogonal when $\ell\neq m$, and, on the other hand, since\footnote{This is the reason why we have introduced an $i$ in the definition of $\tilde{p}_\ell$ from $p_\ell$.} $i\,J^{-1}$ is self-adjoint, $\inp{\tilde{p}_\ell^{0}(\xi,\cdot)}{p_\ell^{0}(\xi,\cdot)}_{L_{per}^2}$, $\ell=1,2$, are real. As a consequence, one may scale the foregoing bases into  $(q_1^{0}(\xi,\cdot),q_2^{0}(\xi,\cdot))$ and $(\tilde{q}_1^{0}(\xi,\cdot),\tilde{q}_2^{0}(\xi,\cdot))$  such that $\tilde{q}_\ell^{0}(\xi,\cdot)\,=\,\varepsilon_\ell\,i\,J^{-1}{q}_\ell^{0}(\xi,\cdot)$, $\ell=1,2$, where $\varepsilon_\ell\in\{-1,1\}$ is the sign of  $\inp{\tilde{p}_\ell^{0}(\xi,\cdot)}{p_\ell^{0}(\xi,\cdot)}_{L_{per}^2}$.

The latter property is preserved by the extension operator to $\delta>0$. This implies that when $\varepsilon_1\,\varepsilon_2>0$, the matrix $D_\xi^\delta$ is skew-adjoint, as a real multiple of
\[
i\,\left(\inp{q_\ell^{\delta}(\xi,\cdot)}{A_{\xi}^{\delta}q_m^{\delta}(\xi,\cdot)}_{L_{per}^2}\right)_{1\leq \ell, m\leq 2}\,.
\]
This recovers in concrete form the above claim about Krein signatures, that are signatures of $i\,J^{-1}$ restricted to characteristic spaces. Concrete computations for \eqref{eep-int} show however that this never happens so that the case to consider is really when $\varepsilon_1\,\varepsilon_2<0$. Then the above normalization puts $D_\xi^\delta$ in the form
\begin{align*}
D_\xi^\delta&=i\,\begin{pmatrix}\alpha_1(\xi,\delta)&\beta(\xi,\delta)\\[0.5em]
-\overline{\beta(\xi,\delta)}&\alpha_2(\xi,\delta)\end{pmatrix}\,,&
\alpha_m(\xi,\delta)\in\mathbb{R}\,,&\quad m=1,2\,.
\end{align*}
The corresponding eigenvalues are
\[
i\,\left(\frac{\alpha_1+\alpha_2}{2}\pm\sqrt{\frac{(\alpha_1-\alpha_2)^2}{4}-|\beta|^2}\right)
\]
(evaluated at $(\xi,\delta)$). Therefore the possible emergence of unstable spectrum near $(\lambda,\xi,\delta)=(\lambda_0,\xi_0,0)$ is effectively reduced to the fact that $(\alpha_1-\alpha_2)^2/4-|\beta|^2$ could take negative values.

Now, on one hand from orthogonality of trigonometric monomials and the fact that in the expansion of profiles as powers of $\delta$ the $\delta^m$-coefficient is a trigonometric polynomial of power less than $m$, one gets that
\[
\beta(\xi,\delta)\stackrel{(\xi,\delta)\to(\xi_0,0)}{=}
\Gamma\,\delta^{|j_1-j_2|}+\cO(\delta^{|j_1-j_2|}\,\|(\xi-\xi_0,\delta)\|)
\]
for some coefficient $\Gamma$. On the other hand one readily checks that if $(\alpha_1-\alpha_2)(\xi,0)$ splits linearly in $\xi-\xi_0$, by the Implicit Function Theorem, there exists a smooth function $\delta\mapsto\Xi(\delta)$ such that $\Xi(0)=\xi_0$ and $(\alpha_1-\alpha_2)(\Xi(\delta),\delta)=0$. When $\p_\xi(\alpha_1-\alpha_2)(\xi_0,0)\neq0$ and $\Gamma\neq 0$, one then concludes that for any $\delta>0$ sufficiently small one gets at $\xi=\Xi(\delta)$ a spectral instability of size at least $\delta^{|j_1-j_2|}$. With a little more work, one may fully describe the shape of the arising instability spectrum and prove that it forms an instability bubble asymptotically shaped as an ellipse. This is illustrated in Figure~\ref{fig:spec}.

\begin{figure}[h]
\begin{center}
 \begin{tabular}{cc}
\includegraphics[width=0.4\textwidth]{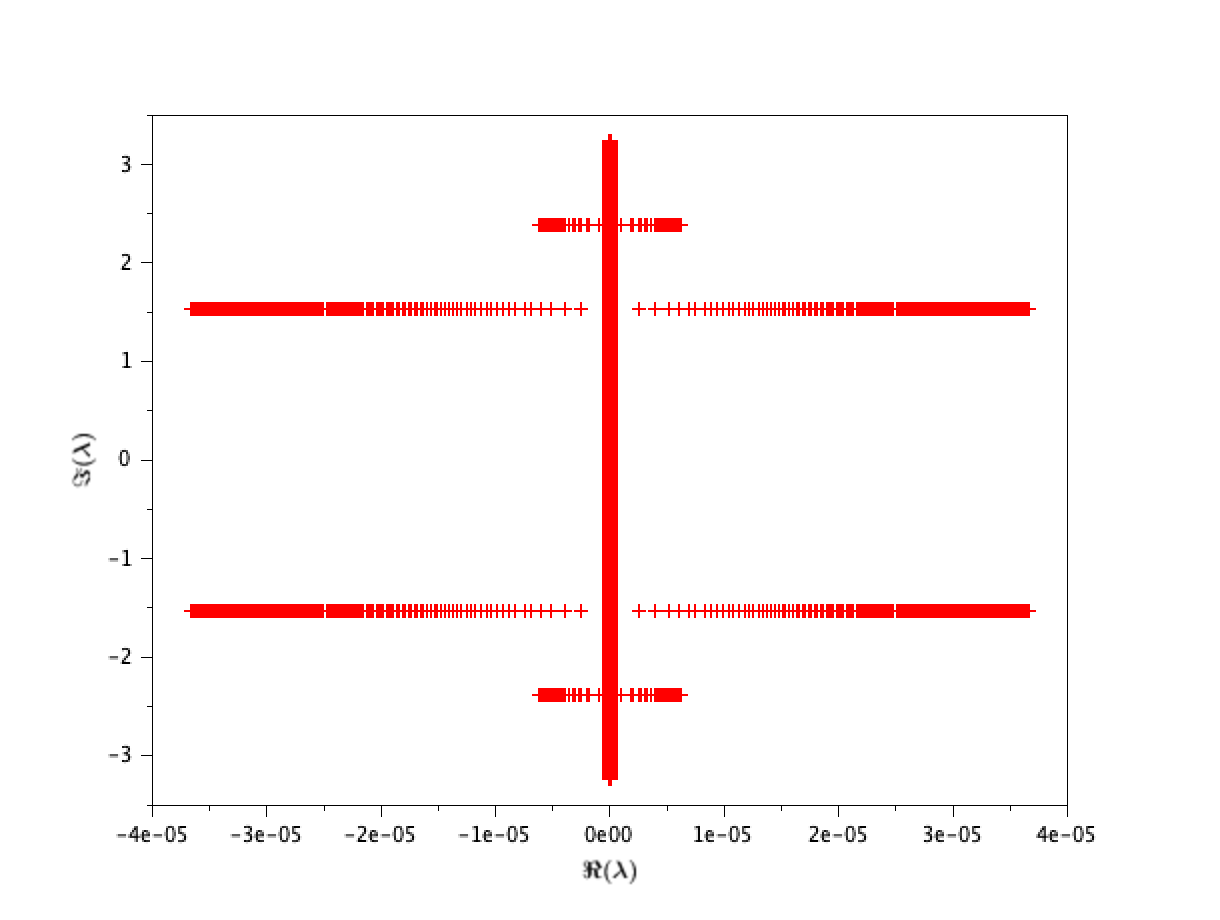} &    
\includegraphics[width=0.4\textwidth]{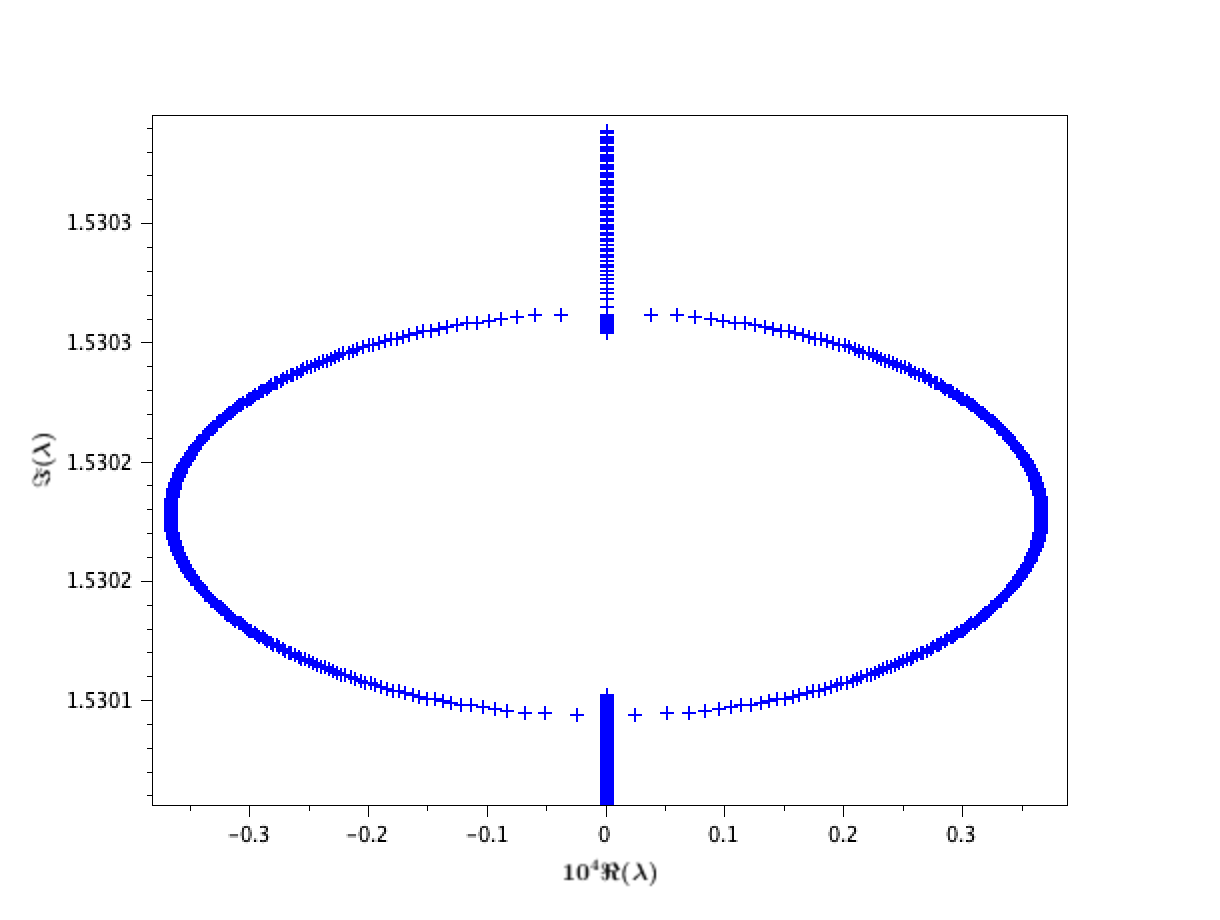} 
\\
(a)  & (b)  \\
\includegraphics[width=0.4\textwidth]{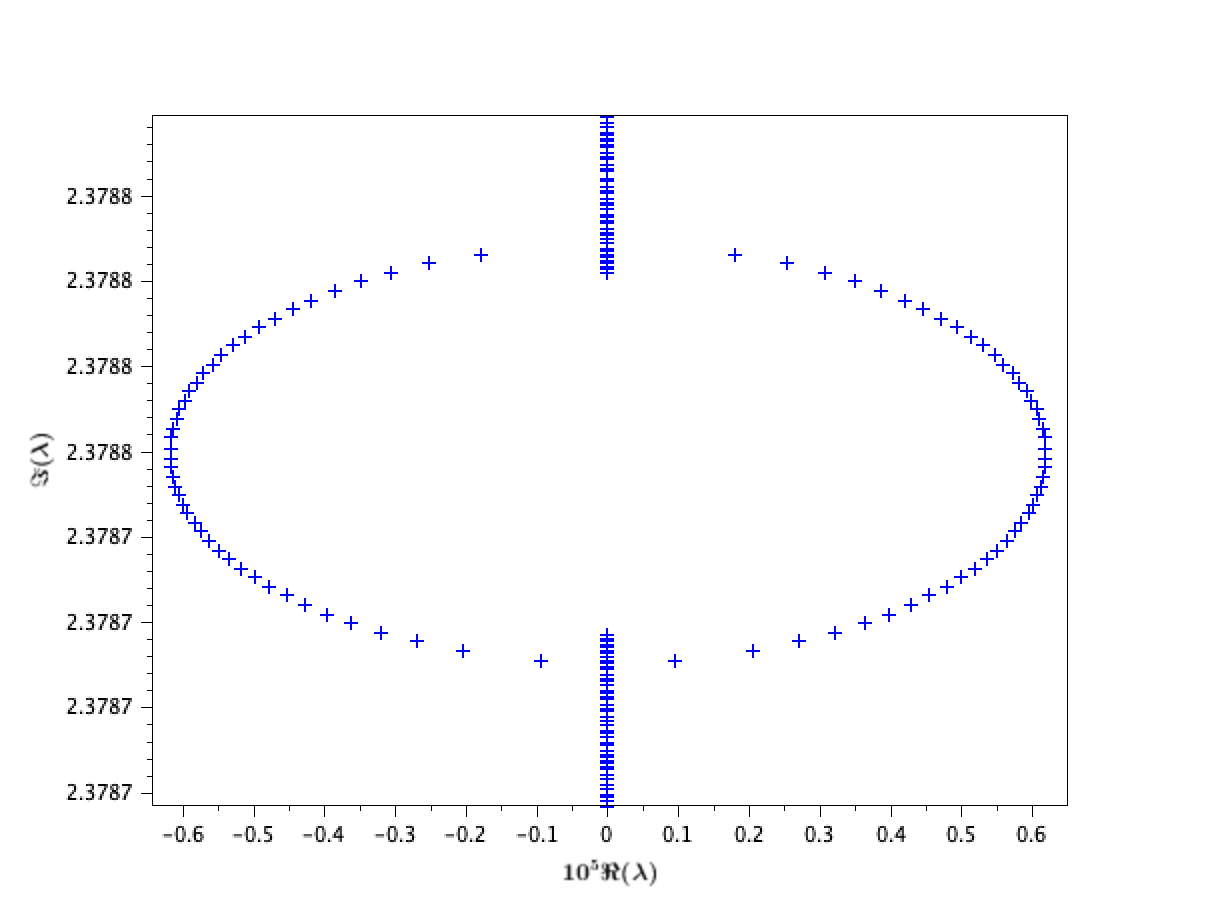} &    
\includegraphics[width=0.4\textwidth]{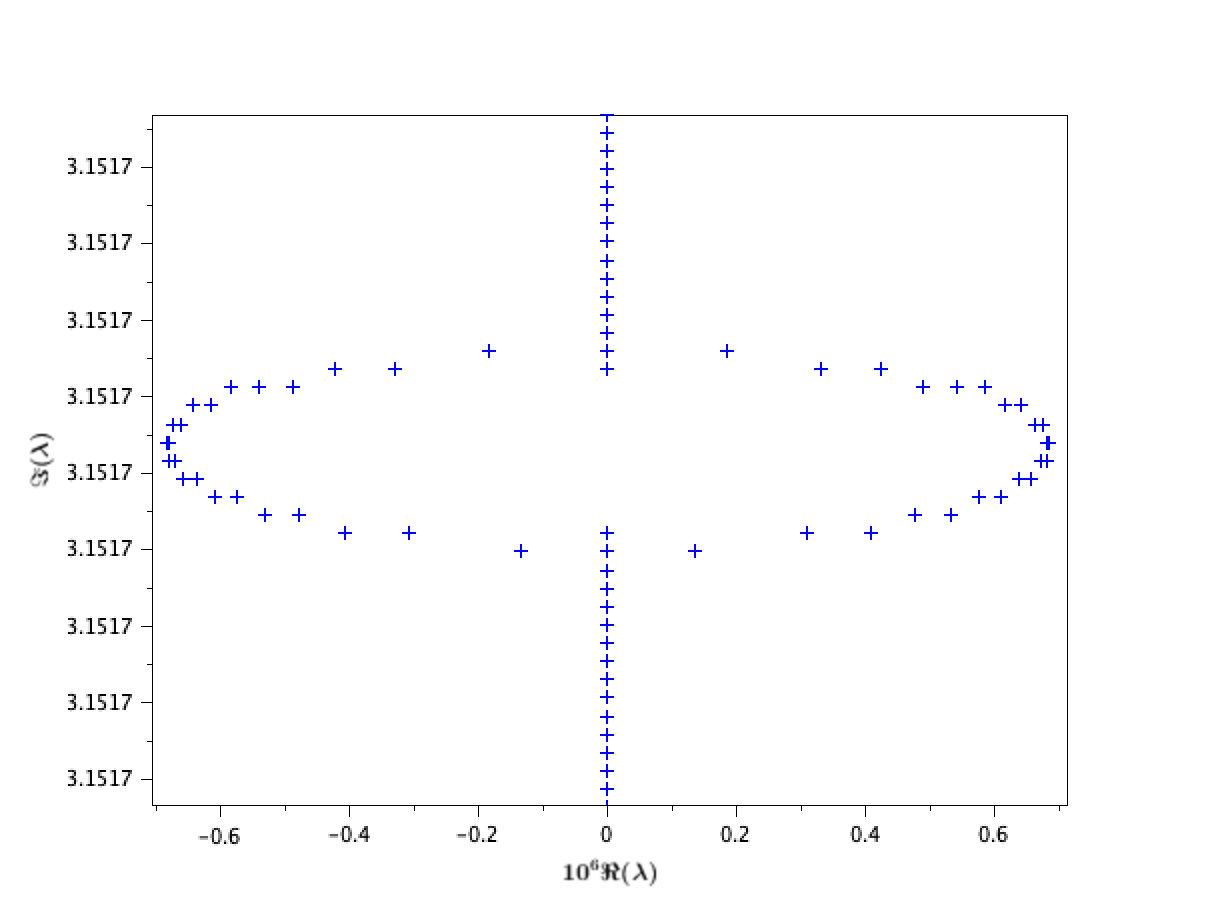} 
\\
(c)  & (d)  
\end{tabular}
\caption{\label{fig:spec}
{\bf An example of spectral instability.} The pressure law is $P(\rho)=\rho^2/4$. The velocity $V$ is such that the limiting wavenumber is $1$. The wave is the one of second smallest amplitude (red curve) in Figure~\ref{fig:profiles}. In (a), the full spectrum is displayed. In (b), (c), and (d), we zoom on the three first instability bubbles, corresponding respectively to $|j_1-j_2|=3$, $|j_1-j_2|=4$ and $|j_1-j_2|=5$. The bubble of (d) is too small to be seen on (a). The spectrum is computed following Hill's method, as detailed in \cite[Section~3.1]{R}.}
 \end{center}
\end{figure}

To complete the discussion there remains mainly to discuss what are possible values of $|j_1-j_2|$. For \eqref{eep-int}, the value of $|j_1-j_2|$ may be any integer larger or equal to $2$. The minimal case $|j_1-j_2|=2$ is reached at $(\lambda_0,\xi_0)=(0,0)$, whose perturbation corresponds to the slow modulational regime, but there $\p_\xi(\alpha_1-\alpha_2)(\xi_0,0)=0$ and another kind of discussion is required, whose conclusion for \eqref{eep-int} is that no instability may arise there, at least for power pressure laws. We stress that the latter vanishing is by no way accidental and is related to the fact that in the small-amplitude limit two characteristic velocities of the modulation system coincide; see the detailed discussion in \cite{Nonlinearity-BMR2}. From the consideration of the next possible value, $|j_1-j_2|=3$, stems Theorem~\ref{main}.

We stress that when all the pieces of the sketched scenario have been carefully justified one may still wish to obtain a concrete expression for the index $\Gamma$. We do so for Theorem~\ref{main} but we warn the reader that despite the fact that this is probably one of the simplest possible computations of its kind, it is still technically demanding. In even more challenging cases of \cite{Hur-Yang,JFM-CDT} similar computations have been carried out using symbolic computation softwares.

\medskip

In the next section, we prove the existence of small-amplitude periodic traveling-wave solutions to \eqref{eep-int} and provide corresponding asymptotic expansion. In Section~\ref{sec3} we provide some elements of background about spectral problem, including elements of Floquet-Bloch theory and constant-coefficient computations. Section~\ref{sec4} is devoted to ruling out slow modulational instabilities, and we show in Appendix~\ref{secW} that this is consistent with formal modulation theory. At last, Section~\ref{sec5} %and its accompanying Appendix~\ref{secA}
proves Theorem~\ref{main}.

\medskip

\noindent {\bf Acknowledgment:} L.M.R. would like to warmly thank Corentin Audiard for enlightening discussions about Krein signatures, during the preparation of \cite{SMF-Audiard-Rodrigues}. L.M.R. expresses his gratitude to INSA Toulouse and C.S. his to Universit\'e de Rennes 1 for their hospitality during part of the preparation of
the present contribution. 
P.N. and L.M.R. would like to thank the Isaac Newton Institute for Mathematical Sciences, Cambridge, for support and hospitality during the programme \emph{Dispersive Hydrodynamics}, where they've heard about \cite{Hur-Yang}. They also thank Casa Matem\'atica Oaxaca for its hospitality during the workshop \emph{Geometrical Methods, non Self-Adjoint Spectral Problems, and Stability of Periodic Structures}, organized by Ram\'on Plaza, where they have heard about an early version of the arguments used in \cite{SIADS-TDK,SIADS-CDT,JFM-CDT}.

\section{\label{sec2}Existence of small-amplitude periodic waves}
%-electron case}

We begin by discussing the existence part of Theorem~\ref{main}. First, we eliminate $\uu^{(\delta,V)}$ through
\begin{equation}\label{uud}
\uu^{(\delta,V)}\,=\,V\,\frac{k^{(\delta,V)}(\ue^{(\delta,V)})'}{1+k^{(\delta,V)}(\ue^{(\delta,V)})'}
\,=\,V\left(1-\frac{1}{1+k^{(\delta,V)}(\ue^{(\delta,V)})'}\right).
\end{equation}
Then, by examining \eqref{profile-intro} near the maximal value of $\ue$ (in the small amplitude regime), one derives that $1$ must be a root of $\mathcal{W}(\cdot;V)$ of even order and that the associated first non-vanishing derivative must be positive. Correspondingly we compute $\p_\rho^2\mathcal{W}(1;V)=h(1;V)=V^2-P'(1)$ so as to derive that $V^2-P'(1)\geq0$ is necessary, and moreover we observe that, when $V^2=P'(1)$, $\p_\rho^3\mathcal{W}(1;V)=-P''(1)-2\,P'(1)$, the latter being nonzero for most common pressure laws, including all convex ones.

For this reason, we assume from now on $V>\sqrt{P'(1)}$. To analyze \eqref{profile-intro}, we point out that by applying a suitable version of the Implicit Function Theorem, as in \cite[Lemma~C.1]{IUMJ-BMR}, one derives that
\[
 \mathcal{W}(1+\zeta\,n;V)=\zeta^2\,\frac{h(1;V)}{2}
\]
is equivalent for $n$ sufficiently close to $1$, respectively $-1$, and $\zeta$  sufficiently small to
\[
n\,=\,\mathcal{W}^\dagger(\zeta;V)\,,
\qquad \textrm{resp. } n\,=\,-\mathcal{W}^\dagger(-\zeta;V)\,,
\]
for some smooth function $\mathcal{W}^\dagger$ taking the value $1$ at $(0,V)$. With this in hands, variations on classical computations of periods of nonlinear pendulums yield that necessarily
\begin{align*}
\frac{1}{k^{(\delta,V)}}
&\,=\,\sqrt{h(1;V)}\,\int_{-1}^{1}\left(\frac{1}{\mathcal{W}^\dagger(\delta\,\sqrt{1-\nu^2};V)}+\frac{1}{\mathcal{W}^\dagger(-\delta\,\sqrt{1-\nu^2};V)}\right)\frac{\d \nu}{\sqrt{1-\nu^2}}\\
%&\,=\,\sqrt{h(1;V)}\,\int_{-\frac{\pi}{2}}^{\frac{\pi}{2}}
%\left(\frac{1}{\mathcal{W}^\dagger(\delta\,\cos(\theta);V)}+\frac{1}{\mathcal{W}^\dagger(-\delta\,\cos(\theta);V)}\right)\ \d \theta\\
&\,=\,\sqrt{h(1;V)}\,\int_{-\pi}^{\pi}
\frac{\d \theta}{\mathcal{W}^\dagger(\delta\,\cos(\theta);V)}\,.
\end{align*}
Note that the dependence of $k^{(\delta,V)}$ on $\delta$ is even.

With this choice of $k^{(\delta,V)}$, it is then sufficient to solve the ODE Cauchy problem
\begin{align*}
(\ue^{(\delta,V)})''&\,=\,-\frac{\ue^{(\delta,V)}}{(k^{(\delta,V)})^2\ h(1+k^{(\delta,V)}\,(\ue^{(\delta,V)})';V)}\,,\\
\ue^{(\delta,V)}(0)&\,=\,0\,,\qquad
(\ue^{(\delta,V)})'(0)\,=\,\delta\,\frac{\mathcal{W}^\dagger(\delta;V)}{k^{(\delta,V)}}\,.
\end{align*}
To motivate the normalization (imposed to quotient the freedom due to the invariance by spatial translations), let us point out that it follows from \eqref{profile-intro} that the maximum and minimum value of $\ue^{(\delta,V)}$ are necessarily opposite. With the present choice, $\ue^{(\delta,V)}$ is odd and $\uu^{(\delta,V)}$ is even.

At this stage, we may already justify one of the claims of the introduction about the structure of the expansion. Indeed one may check recursively that inserting
\[
\ue^{(\delta,V)}\stackrel{\delta\to0}{=}\sum_{j=1}^{M} \delta^j\ue_j^V+ \cO(\delta^{M+1})
\]
in the foregoing profile ODE yields an equation of the form
\[
(\ue_M^V)''+(2\pi)^2\,\ue_M^V\,=\,R_M^V
\]
where $R_M^V$ is a linear combination of $\sin(2\pi\,j\,\cdot)$, $|j|\leq M$, so that $\ue_M^V$ is given by a similar combination.

To complete this short preliminary section, we compute a few coefficients in asymptotic expansions when $\delta\to0$. To begin with, note that
\begin{align*}
k^{(\delta,V)}\,&\stackrel{\delta\to0}{=}\,k_0^V+\delta^2\,k_2^V\,+\cO(\delta^4)\,,&
k_0^V&\,=\,\frac{1}{2\pi\sqrt{h(1;V)}}\,,
%k_2^V&\,=\,-\frac12 k_0^V
%\left((\p_\zeta\mathcal{W}^\dagger(0;V))^2
%-\frac{\p_\zeta^2\mathcal{W}^\dagger(0;V)}{2}\right)\,.
\end{align*}
whereas
\begin{align}\label{ued1}
k_0^V\,\ue_1^{V}
&=\frac{\sin(2\pi\,\cdot)}{2\pi}\,.
\end{align}
Before going on we observe that
\[
\p_\rho^\ell\mathcal{W}(1;V)\,=\,(\ell-1)\ \p_\rho^{\ell-2}h(1;V)\,\qquad \ell\geq2\,,
\]
thus
\begin{align*}
\p_\zeta\mathcal{W}^\dagger(0;V)
&\,=\,-\frac16\frac{\p_\rho^3\mathcal{W}(1;V)}{\p_\rho^2\mathcal{W}(1;V)}
\,=\,-\frac13\frac{\p_\rho h(1;V)}{h(1;V)}\,,\\
\p_\zeta^2\mathcal{W}^\dagger(0;V)
&\,=\,5\,(\p_\zeta\mathcal{W}^\dagger(0;V))^2-\frac{1}{12}\frac{\p_\rho^4\mathcal{W}(1;V)}{\p_\rho^2\mathcal{W}(1;V)}
\,=\,\frac59\left(\frac{\p_\rho h(1;V)}{h(1;V)}\right)^2
-\frac14\frac{\p_\rho^2 h(1;V)}{h(1;V)}\,.
\end{align*}
Therefore,
\begin{align*}
\frac{k_2^V}{k_0^V}&\,=\,\frac12
\left(\frac{\p_\zeta^2\mathcal{W}^\dagger(0;V)}{2}
-(\p_\zeta\mathcal{W}^\dagger(0;V))^2
\right)
\,=\,\frac{1}{12}\left(\frac{\p_\rho h(1;V)}{h(1;V)}\right)^2
-\frac{1}{16}\frac{\p_\rho^2 h(1;V)}{h(1;V)}\,.
\end{align*}
By expanding the Cauchy problem for the profile ODE one also gets
\begin{align*}
k_0^V\,(\ue_2^{V})''+(2\pi)^2\,k_0^V\,\ue_2^{V}&\,=\,
(2\pi)\frac{\p_\rho h(1;V)}{h(1,V)}\,\frac{\sin(4\pi\,\cdot)}{2}\,,\\
\ue_2^{V}(0)&\,=\,0\,,\qquad
(\ue_2^{V})'(0)\,=\,\frac{\p_\rho\mathcal{W}^\dagger(0;V)}{k_0^V}
\,=\,-\frac{\p_\rho h(1;V)}{3\,k_0^V\,h(1;V)}\,,
\end{align*}
thus
\begin{align}\label{ued2}
k_0^V\,\ue_2^{V}
&=-\frac{1}{3}\frac{\p_\rho h(1;V)}{h(1,V)}\,\frac{\sin(4\pi\,\cdot)}{4\pi}\,.
\end{align}
At last,
\begin{align*}
k_0^V\,(\ue_3^{V})''&+(2\pi)^2\,k_0^V\,\ue_3^{V}\\
&\,\in\,
-(2\pi)\left(\frac12\left(\frac{\p_\rho h(1;V)}{h(1,V)}\right)^2-\frac18\frac{\p_\rho^2 h(1;V)}{h(1,V)}\right)\,\sin(6\pi\,\cdot)
+\Span\{\,\sin(4\pi\,\cdot)\,,\,\sin(2\pi\,\cdot)\,\}
\end{align*}
so that
\begin{align}\label{ued3}
k_0^V\,\ue_3^{V}
&\,\in\,
\frac{3}{16}\left(\left(\frac{\p_\rho h(1;V)}{h(1,V)}\right)^2-\frac14\frac{\p_\rho^2 h(1;V)}{h(1,V)}\right)\,\frac{\sin(6\pi\,\cdot)}{6\pi}
+\Span\{\,\sin(4\pi\,\cdot)\,,\,\sin(2\pi\,\cdot)\,\}\,.
\end{align}

To conclude we determine the sign of $k_2^V$ in the power law case.
\begin{lem}
For the pressure law $P=T (\cdot)^{\gamma}$, with $T>0$ and $\gamma\geq 1,$ we have
\beq\label{signa2}
k_2^V>0.
\eeq
\end{lem}
\begin{proof}
This stems from $V^2>T\gamma$ and the following direct computations,
\begin{align*}
h(1;V)&=V^2-T\gamma\,,\qquad \p_\rho h(1;V)=-3 V^2-T\gamma(\gamma-2)\,,\\
\p_\rho^2 h(1;V)&=12 V^2-T\gamma(\gamma-2)(\gamma-3)\,.
\end{align*}
Indeed those imply
\begin{align*}
-\p_\rho^2 h(1;V)h(1;V)&+\f{4}{3}(\p_\rho h(1;V))^2\\
&=V^2 T\gamma  (\gamma+1)(\gamma+2)+ (T\gamma)^2\left(\f{4}{3}(\gamma-2)^2-(\gamma-2)(\gamma-3)\right)\\
&>(T\gamma)^2 \left((\gamma+1)(\gamma+2)-(\gamma-2)(\gamma-3)+\f{4}{3}(\gamma-2)^2\right)
\,=\,\frac43\,(T\gamma)^2 (\gamma+1)^2.
\end{align*}
\end{proof}

For the sake of readability, from now on we shall drop marks of dependencies on $V$, since the role of $V$ is more passive than the one of $\delta$. Correspondingly we shall denote with primes the derivatives with respect to $\rho$ of functions of $(\rho;V)$.

Before going on with our analysis, we would like to stress that the family of waves built here can be continued beyond the small amplitude limit. It does terminate though, not as a solitary wave, since the system does not admit solitary waves, but as a peakon. Explicitly, when the pressure law is taken to be an increasing and strictly convex function, $h$ vanishes exactly once at some value $\rho_{max}$, and $\rho_{max}>1$. This value is the maximal value for the electronic charge density of waves and the family of periodic waves ends when this value is reached. The limiting object is a periodic peakon in orginal variables $(\rho,u)$, that is, in these variables the limiting wave is periodic, continuous, piecewise $\cC^1$, but has a jump in its first-order derivative. This jump takes place where $\rho=\rho_{max}$ and can be easily computed. We illustrate this phenomenon in Figure~\ref{fig:profiles}. We leave as a possible interesting further development the elucidation of the peakon instability and its use for smooth waves of nearly maximal amplitude.%}

\begin{figure}[h]
\begin{center}
\includegraphics[width=0.8\textwidth]{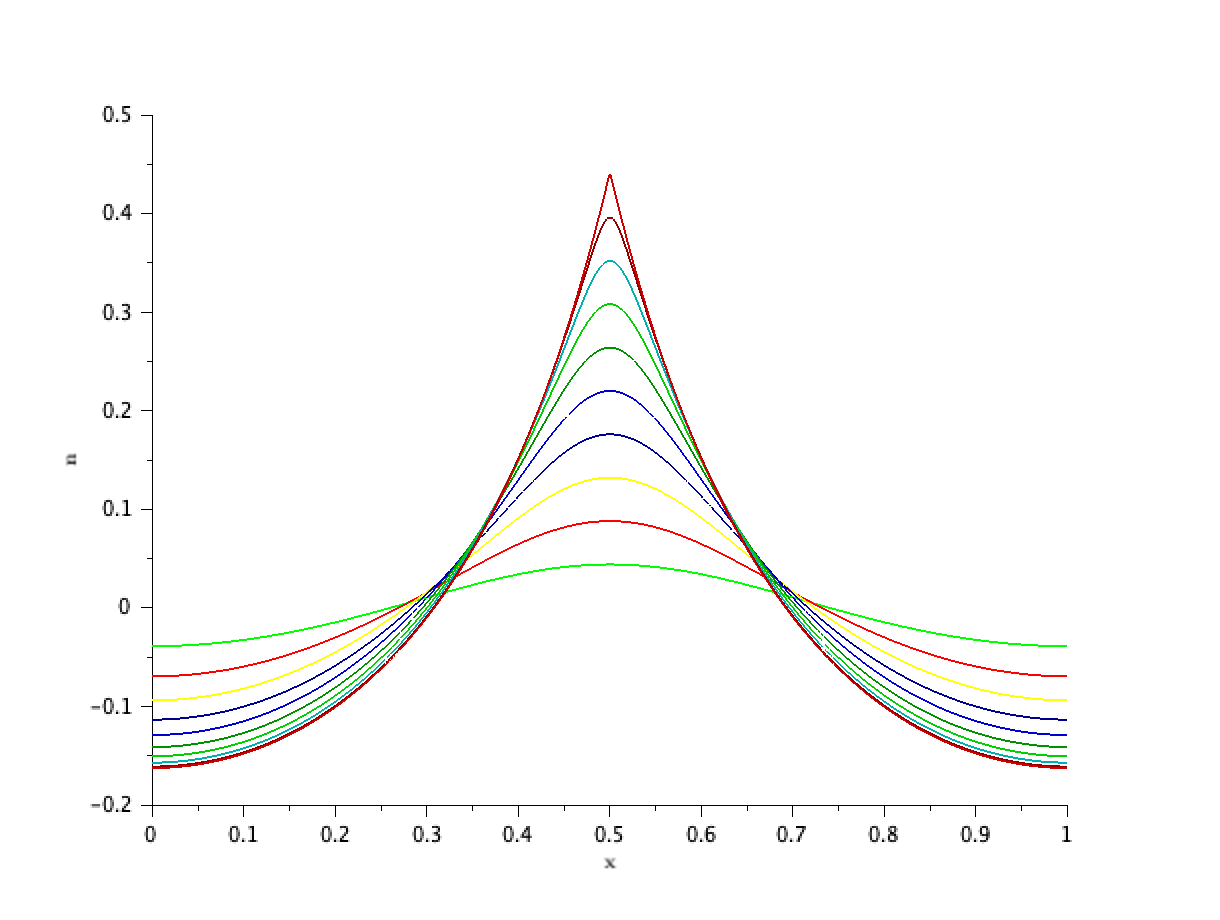} 
\caption{\label{fig:profiles}
{\bf A family of wave profiles.} Plot of the scaled profiles for the total charge $\underline{n}=\underline{\rho}-1=k\underline{E}'$ of various waves. The pressure law is $P(\rho)=\rho^2/4$. The velocity is held fixed to the value $V$ such that $k_0^V=1$.}
 \end{center}
\end{figure}

\section{Spectral preliminaries}\label{sec3}

To benefit from the structure of nearby periodic waves, it is useful not only to go to a co-moving frame, so as to make the wave stationary, but also to scale period. Namely, here, we analyze solutions $(E,u)$ to System~\eqref{eep-int} through $(\tilde E,\tilde u)$ defined by\footnote{We recall that $k^{(\delta)}$ also depends on $V$ but we do not mark this dependency anymore.}
\[
(E,u)(t,x)\,=\,(\tilde E,\tilde u)(t,k^{(\delta)}\,(x-t\,V))\,.
\]
However for concision's sake, we drop tildes and simply observe that after this change of coordinates linearizing about $(\ue^\delta,\uu^\delta)$ yields,
\begin{align*}
\p_tE&+(\uu^\delta-V)\,k^{(\delta)}\p_x\,E
+(1+k^{(\delta)}(\ue^{\delta})')\,u\,=\,0\,,\\
\p_tu&+k^{(\delta)}\p_x\left((\uu^{\delta}-V)\,u\right)
+k^{(\delta)}\p_x\left(F''(1+k^{(\delta)}(\ue^{\delta})')\,k^{(\delta)}\p_xE\right)
\,=\,E\,.
\end{align*}
The latter is also written, for $U=(E,u)$, as
\[
\p_t U\,=\,L^\delta U\,,
\]
where $L^\delta=\Sigma^L_\delta(x,k^{(\delta)}\p_x)$ is a $1$-periodic differential operator, frequency-scaled from $\Sigma^L_\delta(x,\p_x)$ the differential operator associated through standard\footnote{Not Weyl's.} quantization with symbol $\Sigma^L_\delta$
\begin{align}%\label{defLdelta}
%&
\Sigma^L_\delta(\cdot,\zeta)%\\
\nonumber
&:=\begin{pmatrix}
(V- \uu^{\delta})\,\zeta&-(1+k^{(\delta)}(\ue^{\delta})')\\
1-F''(1+k^{(\delta)}(\ue^{\delta})')\zeta^2
-F'''(1+k^{(\delta)}(\ue^{\delta})')\,(k^{(\delta)})^2(\ue^{\delta})''\,\zeta&
(V- \uu^{\delta})\zeta-k^{(\delta)}(\uu^{\delta})'
    \end{pmatrix}.
\end{align}

\subsection{Functional-analytic framework}

It is important to note that from $V>\sqrt{P'(1)}$ stems that, when $\delta$ is sufficiently small,
\[
(V- \uu^{\delta})^2-F''(1+k^{(\delta)}(\ue^{\delta})')(1+k^{(\delta)}(\ue^{\delta})')
=(V- \uu^{\delta})^2-P'(1+k^{(\delta)}(\ue^{\delta})')
\]
does not vanish so that the operator $L^\delta$ is non characteristic. As a consequence, considering $L^\delta$ as acting on $H^1(\mathbb{R})\times L^2(\mathbb{R})$ with domain $H^2(\mathbb{R})\times H^1(\mathbb{R})$, turns it into a densely defined closed operator. Moreover, from the fact that the same quantity is positively-valued one deduces that the linearized system is symmetrizable, which implies, as expected, that $L^\delta$ does generate a $\cC^0$-semigroup on $H^1(\mathbb{R})\times L^2(\mathbb{R})$.

Note that $L^\delta$ has real coefficients, thus it commutes with complex conjugation so that its spectrum is symmetric with respect to the real axis.

Furthermore, $L^\delta$ has the Hamiltonian structure
\[
L^\delta\,=\,J\,A^\delta
\]
with skew-adjoint operator $J$ given by
\begin{align*}
J\,&=\,\begin{pmatrix}0&-1\\1&0\end{pmatrix}\,,
\end{align*}
and $A^\delta$ a self-adjoint operator, explicitly, $A^\delta=\Sigma^A_\delta(x,k^{(\delta)}\p_x)$ where
\begin{align*}
\Sigma^A_\delta(\cdot,\zeta)
\nonumber
&:=\begin{pmatrix}
1-F''(1+k^{(\delta)}(\ue^{\delta})')\zeta^2
-F'''(1+k^{(\delta)}(\ue^{\delta})')\,(k^{(\delta)})^2(\ue^{\delta})''\,\zeta&
(V- \uu^{\delta})\zeta-k^{(\delta)}(\uu^{\delta})'\\
(\uu^{\delta}-V)\,\zeta&(1+k^{(\delta)}(\ue^{\delta})')\\
    \end{pmatrix}.
\end{align*}
Up to scaling, $A^\delta$ is the variational Hessian of $\mathcal{H}+V\,\mathcal{M}$ at the background wave, where $\mathcal{M}$ is the momentum density,
\[
\mathcal{M}[E,u]:=-\,u\,\p_xE\,,
\]
or, in other words, where $\mathcal{M}$ generates spatial translations in the sense that $J\delta\mathcal{M}[E,u]=\p_x(E,u)$. As a consequence, one obtains the following relation between $L^\delta$ and its adjoint\footnote{Throughout the paper we use Hilbertian formalism for adjoints, that is, we identify Hilbert spaces with their duals when considering adjoints so that operators and their adjoints act on the same spaces.} $(L^\delta)^*$,
\[
(L^\delta)^*=-J^{-1}L^\delta J\,.
\]
In particular, the spectrum of $L^\delta$ is symmetric with respect to the imaginary axis.

We recall that one of the main consequences of the latter is that the traveling wave under consideration is spectrally stable if and only if the spectrum of $L^\delta$ is contained in the imaginary axis.

\subsection{Bloch transform}

To analyze the spectrum of periodic coefficient operators it is expedient to introduce Bloch symbols, associated with the (Floquet-)Bloch transform. Thus we recall here a few basic elements of the corresponding theory.

Let us first introduce notation $\cF(f)=\hat{f}$ for the Fourier transform of $f$, defined explicitly when $f\in L^1(\mathbb{R})$ by
\begin{align*}
\cF(f)(\xi)=\hat{f}(\xi):=\f{1}{2\pi}\int_{\mR}e^{-i\xi x} f(x)\d x
\end{align*}
and recall that when both $f\in L^1(\mathbb{R})$ and $\hat f\in L^1(\mathbb{R})$
\[
f(x)\,=\,\int_{\mR} e^{ix\xi} \hat{f}(\xi)\,\d \xi
\]
holds in pointwise sense. The latter is extended by density and duality to less smooth and less localized $f$, using mostly the classical $L^2$ isometry of the Fourier transform.

Likewise, the inverse Bloch transform provides for any $f\in L^1(\mathbb{R})$ such that $\hat f\in L^1(\mathbb{R})$,
\begin{align}\label{def-bloch}
\displaystyle
    f(x)=\int_{-\pi}^{\pi} e^{ix\xi} \check{f}(\xi,x)\,\d \xi
\end{align}
where the direct Bloch transform of $f$, $\cB(f)=\check{f}$, is defined as
\begin{align}\label{def-Bloch}
\cB(f)(\xi,x)=\check{f}(\xi,x):=\sum_{j\in \mathbb{Z}} e^{i2\pi jx } \hat{f}(\xi+2j\pi)
\,=\,\sum_{\ell\in \mathbb{Z}} e^{-i\xi\,(x+\ell)} f(x+\ell)\,.
\end{align}
Note that by design, for each $\xi$, $\check{f}(\xi,\cdot)$ is periodic of period $1$, so that $f$ is written as an integral of functions $g_\xi\,=\,e^{i\xi\,\cdot} \check{f}(\xi,\cdot)$ satisfying $g_\xi(x+1)=e^{i\xi}\,g_\xi(x)$. Such functions are known as Bloch waves, the number $e^{i\xi}$ being a Floquet multiplier, and $\xi$ is classically called Floquet exponent. Various extensions are then carried out by exploiting that $\sqrt{2\pi}\,\cB$ is an isometry from $L^2(\mR)$ to $L^2((-\pi,\pi);L^2(\mR/\mZ))$.

With the help of the inverse Bloch transform \eqref{def-bloch}, one can turn an operator with periodic coefficients acting on a function over $\mR$ into a family of operators, its Bloch symbols, acting on periodic functions. Explicitly, for the case at hand,
\begin{align}
L^{\delta}(U)(x)=\int_{-\pi}^{\pi} e^{ix\xi} L_{\xi}^{\delta}(\cB{U}(\xi, \cdot))(x) \d \xi.
\end{align}
where each $L_\xi^\delta:=\Sigma^L_\delta(x,k^{(\delta)}(\p_x+i\xi))$ acts on $H^1(\mathbb{R}/\mZ)\times L^2(\mathbb{R}/\mZ)$ with domain $H^2(\mathbb{R}/Z)\times H^1(\mathbb{R}/Z)$. The main gain when replacing the direct analysis of $L^\delta$ with those of $L_\xi^\delta$ is that each $L_\xi^\delta$ has compact resolvent thus spectrum reduced to eigenvalues of finite multiplicity, arranged discretely. A key related observation is that
\begin{align*}
    \sigma(L^{\delta})=\bigcup_{\xi\in[-\pi,\pi]}\sigma_{per}(L_{\xi}^{\delta}),
\end{align*}
where we have added the suffix $per$ to mark that each $L_{\xi}^{\delta}$ acts on functions over $\mR/\mZ$. The latter decomposition is by now classical in the field and we refer the reader for instance to \cite[Appendix~A]{Mielke} or \cite[p.30-31]{R} for a proof.

To conclude, we point out how real and Hamiltonian symmetries are transferred to Bloch symbols, namely
\begin{align*}
\overline{L_\xi^\delta\,U}&=L_{-\xi}^\delta\,\overline{U}\,,&
(L_\xi^\delta)^*&=-J^{-1}L_\xi^\delta J\,.
\end{align*}

\subsection{Constant-coefficient computations}

Since our analysis is perturbative from the constant-coefficient case --- when $\delta=0$ ---, we need to recast some of the classical Fourier constant-coefficient computations in terms of Floquet-Bloch analysis.

Since $\Sigma^L_0(x,\zeta)$ does not depend on $x$, let us accordingly denote by $\Sigma^L_0(\zeta)$ its value. By using spectral decomposition arising from the introduction of Fourier series, one obtains
\[
\sigma_{per}(L_{\xi}^0)
=\bigcup_{j\in \mathbb{Z}}\,\sigma\left(\Sigma^L_0\left(k^{(0)}(2\pi\,j+\xi)\right)\right)
=\bigcup_{j\in \mathbb{Z}} \{\lambda_{-}^{j}(\xi);\lambda_{+}^{j}(\xi)\}
\]
where
\begin{align}\label{eigenvalue}
\lambda_{\pm}^{j}(\xi)&:=\lambda_{\pm}(k^{(0)}\,(2\pi j+\xi))\,,&
\lambda_{\pm}(\zeta)&:=i \left(V\,\zeta\pm \sqrt{1+P'(1)\zeta^2}\right)\,.
\end{align}
Note incidentally that combining the latter Fourier series spectral decomposition with the Bloch transform spectral decomposition recovers the classical Fourier transform spectral decomposition
\[
\sigma(L^0)
=\bigcup_{\zeta\in\mR}\,\sigma\left(\Sigma^L_0\left(\zeta\right)\right)
=\bigcup_{\zeta\in \mR} \{\lambda_{-}(\zeta);\lambda_{+}(\zeta)\}\,.
\]
For later use, we also introduce notation $\varphi_{\pm}^{j}(\xi, \cdot)$ for Bloch eigenfunctions associated with $\lambda_{\pm}^{j}(\xi)$,
\begin{align}\label{defbasis-constant}
\varphi_{\pm}^{j}(\xi,x)
:=e^{i\,2\pi j\,x}\begin{pmatrix}1\\\mp\,i\omega_{j}(\xi)\end{pmatrix}
\end{align}
where
\begin{equation}\label{defomega}
\omega_{j}(\xi):=\sqrt{1+ P'(1) (k^{(0)})^2(2\pi j+\xi)^2}.
\end{equation}

It follows readily from Hamiltonian symmetry that unstable spectra may only arise from multiple eigenvalues. With this in mind, we observe that from $V>\sqrt{P'(1)}$ stems that the functions $\zeta\mapsto \Im(\lambda_\pm(\zeta))$ are both strictly increasing (with derivative bounded away from zero). Therefore multiplicity is at most two and we only need to investigate the existence of $(j,j',\xi)$ such that $\lambda_-^j(\xi)=\lambda_+^{j'}(\xi)$. This is the purpose of the following lemma.
\begin{lem}\label{l:0-lemma}
\begin{enumerate}
\item For any  $(j,j',\xi)\in\mZ^2\times[-\pi,\pi]$ such that $j'\geq j-1$, $\lambda_-^j(\xi)\neq\lambda_+^{j'}(\xi)$.
\item There exists a unique $(j_2,j_2',\xi_2)\in\mZ^2\times[-\pi,\pi)$ such that $j_2'=j_2-2$ and $\lambda_-^{j_2}(\xi_2)=\lambda_+^{j_2'}(\xi_2)$.
\item For any $\ell\in\mN$, $\ell> 2$, there exist exactly two $(j_{\pm,\ell},j_{\pm,\ell}',\xi_{\pm,\ell})\in\mZ^2\times[-\pi,\pi)$ such that $j_{\pm,\ell}'=j_{\pm,\ell}-\ell$ and $\lambda_-^{j_{\pm,\ell}}(\xi_{\pm,\ell})=\lambda_+^{j_{\pm,\ell}'}(\xi_{\pm,\ell})$.
\end{enumerate}
Moreover, $j_2=1$, $j_{2}'=-1$, $\xi_2=0$ and one may normalize with
\begin{align}\label{para-jk}
j_{\pm,\ell}&=\left\lfloor\f{1}{2}\bigg(\ell+1\pm \f{\sqrt{\ell^2-4}}{\sqrt{P'(1)}/V}\bigg)\right\rfloor,&
j'_{\pm,\ell}&=j_{\pm,\ell}-\ell,&
\f{\xi_{\pm,\ell}}{2\pi}&=\left\{\f{1}{2}\bigg(\ell+1\pm \f{\sqrt{\ell^2-4}}{\sqrt{P'(1)}/V}\bigg)\right\}-\f{1}{2}
\end{align}
with $\lfloor\cdot\rfloor$ denoting the least integer part and $\{\cdot\}$ being the associated fractional part.
\end{lem}

The proof of the lemma is elementary but relies on arguments far from those of the rest of the paper so that we have postponed it to Appendix~\ref{s:0}.

Note that, except for the crossing associated with $\ell=2$, crossings occur at non zero eigenvalues and
\begin{align*}
j_{+,\ell}+\frac{\xi_{+,\ell}}{2\pi}&=
-\left(j'_{-,\ell}+\frac{\xi_{-,\ell}}{2\pi}\right)\,,&
j'_{+,\ell}+\frac{\xi_{+,\ell}}{2\pi}&=
-\left(j_{-,\ell}+\frac{\xi_{-,\ell}}{2\pi}\right)\,,
\end{align*}
so that crossings associated with the same frequency gap $\ell\geq 3$ are obtained from the other by using real symmetry. In particular, from the point of view of stability, by using real symmetry, one may reduce studies to crossings labeled by $(j_2,j_2',\xi_2)$ and $(j_{+,\ell},j_{+,\ell}',\xi_{+,\ell})$, $\ell\geq 3$.

In the rest of the paper we study how $\ell=2$ and $\ell=3$ crossings perturb.

\section{Slow modulation stability}\label{sec4}

We begin with the $\ell=2$ crossing, that is, we study the spectrum of $L^\delta_\xi$ near the origin when $(\delta,\xi)$ is small. Our main conclusion is summarized in the following result.

\begin{thm}\label{th4}
Under the assumptions of Theorem~\ref{main}, with
\[
I_{mod}:=\left\{\,V\in(\sqrt{P'(1)},+\infty)\,;\,k_2^V>0\,\right\}
\]
there exists a smooth function $\epsilon_{mod}:I_{mod}\to \mathbb{R}_+^*$ such that when $V\in I_{mod}$, for any $0\leq\delta\leq \epsilon_{mod}(V)$, $|\xi|\leq \epsilon_{mod}(V)$,
\[
\displaystyle
\sigma\left(L_{\xi}^{\delta}\right)\cap B_{\epsilon_{mod}(V)}(0)\subset i\mathbb{R}\,.
\]
\end{thm}

In the foregoing statement and throughout the text $B_r(\gamma_0)$ denotes the open ball of the complex plane centered at $\gamma_0$ with radius $r$.

Our proof also shows that $k_2^V<0$ implies instability for sufficiently small-amplitude waves; see Remark~\ref{rk:BF}. But since for power laws, $I_{mod}=(\sqrt{P'(1)},+\infty)$, we do not give details on the latter.

The result covers a region larger than obtained by taking the small-amplitude limit of the slow modulational regime, as in \cite{Nonlinearity-BMR2,SMF-Audiard-Rodrigues}. The latter corresponds to first, holding $(\delta,V)$ fixed, consider small spectral and Floquet parameters, then sending $\delta$ to $0$ in the obtained criterion. One of the interesting features of the slow modulational regime is that it is connected to geometrical optics \`a la Whitham, so that its conclusions may be guessed by arguing heuristically. For the convenience of the reader, we provide some elements of this formal analysis in Appendix~\ref{secW}.

\subsection{Finite-dimensional reduction}\label{s:reduc-mod}

Our first step is to provide a reduction to the consideration of the spectrum of a $2\times 2$ matrix parametrized by $(\delta,\xi)$.

For comparison, note that the starting point of the rigorous mathematical spectral analysis of the slow modulational regime is that for any $\delta>0$ the spectrum of $L_0^\delta$ near the origin is reduced to $\{0\}$ and the associated structure is described\footnote{Under generic assumptions, satisfied here for small-amplitude waves.} in terms of the structure of the family of traveling waves. Note in particular that the spectrum does not move when one varies $\delta$ but holds $\xi$ fixed to zero. For this reason, it is convenient even for our analysis where no size comparison is assumed between $|\xi|$ and $\delta$ to organize the computation as in \cite{JNS-BNR,Nonlinearity-BMR2,SMF-Audiard-Rodrigues}.

Thus, as in these references, we begin by gathering what may be obtained by differentiating wave profile ODEs with respect to parameters, namely,
\begin{align*}
L_0^{\delta} ((\uU^{\delta})')&=0\,,\\
L_0^{\delta} (\p_{\delta}\uU^{\delta})&=-\p_{\delta}k^{(\delta)} L_{[1]}^{\delta} (\uU^\delta)'\,,\\
L_0^{\delta} (\p_{V}\uU^{\delta})&=-k^{(\delta)} (\uU^{\delta})'-\p_{V}k^{(\delta)}L_{[1]}^{\delta} (\uU^{\delta})'
\end{align*}
where we have introduced notation for the following Floquet expansion
\begin{align*}
L_{\xi}^{\delta}=L_0^{\delta}+ i\,k^{(\delta)}\xi\ L_{[1]}^{\delta}
+(i\,k^{(\delta)}\xi)^2 L_{[2]}^{\delta}\,.
\end{align*}
Consistently with the above claim one obtains a Jordan chain of size $2$ for the eigenvalue $0$ by combining $(\uU^{\delta})'$ and a suitable combination of $\p_{\delta}\uU^{\delta}$ and $\p_{V}\uU^{\delta}$. In view of expansions from Section~\ref{sec2}, in order to obtain smooth non trivial limits when $\delta\to0$, we set
\begin{align}\label{defbases-0}
\displaystyle
q_1^{\delta}(0,\cdot)&:=\f{(\uU^{\delta})'}{\delta}\,,&
q_2^{\delta}(0,\cdot)&:= \p_{\delta} \uU^{\delta} -\f{\p_{\delta} k^{(\delta)}}{\p_{V}k^{(\delta)}}\p_{V}\uU^{\delta}
\end{align}
so that
\begin{align}\label{usefulid}
L_0^{\delta}q_1^{\delta}(0,\cdot)&=0,& L_0^{\delta}q_2^{\delta}(0,\cdot)&=\delta  k^{(\delta)}\f{\p_{\delta} k^{(\delta)}}{\p_{V}k^{(\delta)}}q_1^{\delta}(0,\cdot)
\end{align}
and, at the limit $\delta=0$,
\begin{align}\label{basis-0xi}
q_1^{0}(0,x)&=2\pi \left( \begin{array}{c}
     \sqrt{h(1)} \cos (2\pi x) \\
          -V \sin(2\pi x)
    \end{array}\right)\,,&
q_2^{\delta}(0,x)&=  \left( \begin{array}{c}
     \sqrt{h(1)} \sin (2\pi x) \\
          V \cos(2\pi x)
    \end{array}\right)\,.
\end{align}

Therefore for some $\epsilon_0>0$, when $0\leq\delta<\epsilon_0$, $(q_1^{\delta}(0,\cdot),q_2^{\delta}(0,\cdot))$ form a basis of the sum of characteristic spaces of $L_0^\delta$ associated with eigenvalues in $B_{\epsilon_0}(0)$, and in this basis, $L_0^\delta$ restricted to this space has matrix
\begin{align}\label{D0delta}
    D_{0}^{\delta}:=\left(\begin{array}{cc}
        0 &  \delta  k^{(\delta)}\f{\p_{\delta} k^{(\delta)}}{\p_{V}k^{(\delta)}} \\
       0  & 0
    \end{array}\right)
\end{align}
where
\begin{align*}
\delta  k^{(\delta)}\f{\p_{\delta} k^{(\delta)}}{\p_{V}k^{(\delta)}}
&\stackrel{\delta\to0}{=}c_{02}\delta ^2+\cO(\delta^3)\,,&
c_{02}:=-\f{k_2^V\,(V^2-P'(1))}{V}<0\,,
\end{align*}
provided that $k_2^V>0$ as we have assumed.

We want to extend these computations to small nonzero $\xi$. This may be carried out by extending the basis $(q_1^{\delta}(0,\cdot),q_2^{\delta}(0,\cdot))$ following Kato's perturbation theory, especially \cite[p.99-100]{Kato}. To begin with, when $(\xi,\delta,\epsilon_0)$ is sufficiently small, we may define
\begin{align}\label{def-projector}
    \Pi_{\xi}^{\delta}:=\f{1}{2i \pi} \oint_{\p (B_{\epsilon_0}(0))}
    (\lambda I-L_{\xi}^{\delta})^{-1} \d \lambda.
\end{align}
the spectral projector of $L_\xi^\delta$ on the sum of characteristic spaces associated with eigenvalues in $B_{\epsilon_0}(0)$. An extension operator $\mathcal{U}^\delta({\xi})$ is then obtained by solving the Cauchy problem
\begin{align}\label{def-extop}
\p_\xi\mathcal{U}^{\delta}(\xi)
&=[\p_\xi\Pi_{\xi}^{\delta}, \Pi_{\xi}^{\delta}]\ \mathcal{U}^{\delta}(\xi)\,,
&\mathcal{U}^{\delta}(0)=I\,,
\end{align}
where $[\cdot,\cdot]$ denotes the commutator,  $[A,B]:=AB-BA$. Setting
\begin{align*}
q_j^\delta(\xi,\cdot)&:=\mathcal{U}^{\delta}(\xi) q_j^{\delta}(0,\cdot)\,,&
j=1,2,
\end{align*}
does provide a basis of $\Ran(\Pi_{\xi}^{\delta})$. To obtain a matrix for the action of $L_\xi^\delta$ in the corresponding basis, it is convenient to introduce a dual basis, associated with the spectrum of $(L_\xi^\delta)^*$ in $B_{\epsilon_0}(0)$, that is, spanning $\Ran(\Pi_{\xi}^{\delta})^*$.

We first provide a dual basis when $\xi=0$, so as to extend it later. By Hamiltonian symmetry, $(J^{-1}\,q_1^{\delta}(0,\cdot),J^{-1}\,q_2^{\delta}(0,\cdot))$ spans $\Ran(\Pi_{0}^{\delta})^*$. By skew-symmetry of $J$,
\begin{align*}
\langle J^{-1} q_j^{\delta}(0,\cdot),q_j^{\delta}(0,\cdot)\rangle&=0\,,\quad j=1,2,&
\langle J^{-1} q_1^{\delta}(0,\cdot),q_2^{\delta}(0,\cdot)\rangle&=
-\langle J^{-1} q_2^{\delta}(0,\cdot),q_1^{\delta}(0,\cdot)\rangle\,,
\end{align*}
whereas a direct computation gives
\begin{align*}
\langle J^{-1} q_2^{0}(0,\cdot),q_1^{0}(0,\cdot)\rangle&=V\,2\pi\,\sqrt{h(1)}\neq 0\,,
\end{align*}
where $\langle\cdot,\cdot\rangle$ denotes $L^2(\mR/\mZ)^2$ scalar product,
\[
\langle f,g\rangle:=\int_0^1 \overline{f}\cdot g\,.
\]
Thus
\begin{align*}%\label{defdualbase}
    \tilde{q}_1^{\delta}(0,\cdot)&:=\f{J^{-1} q_2^{\delta}(0,\cdot)}{\langle J^{-1} q_2^{\delta}(0,\cdot), q_1^{\delta}(0,\cdot)\rangle}\,,&
    \tilde{q}_2^{\delta}(0,\cdot)&:=\f{-J^{-1} q_1^{\delta}(0,\cdot)}{\langle -J^{-1} q_1^{\delta}(0,\cdot), q_2^{\delta}(0,\cdot)\rangle}\,,
\end{align*}
provides the sought basis for $\xi=0$. Then one gets the required $(\tilde{q}_1^{\delta}(\xi,\cdot),\tilde{q}_2^{\delta}(\xi,\cdot))$ through
\[
\tilde{q}_j^{\delta}(\xi,\cdot):=
\mathcal{V}^\delta(\xi)\,\tilde{q}_j^\delta(\xi,\cdot)\,,
\qquad j=1,2,
\]
where the dual extension operator $\mathcal{V}^\delta(\xi)$ is obtained from the Cauchy problem
\begin{align*}
\p_\xi\mathcal{V}^{\delta}(\xi)
&=[\p_\xi(\Pi_{\xi}^{\delta})^*,(\Pi_{\xi}^{\delta})^*]\ \mathcal{V}^{\delta}(\xi)\,,
&\mathcal{V}^{\delta}(0)=I\,.
\end{align*}

Note that the general construction gives $(\mathcal{V}^{\delta}(\xi))^*\,\mathcal{U}^{\delta}(\xi)=I$, for any small $\xi$, which is the key to preserve duality relations. Moreover the Hamiltonian symmetry also yields for any small $\xi$
\[
(\Pi_\xi^\delta)^*=-J^{-1}\,\Pi_\xi^\delta\,J
\]
thus for any small $\xi$
\[
\mathcal{V}^{\delta}(\xi)\,=\,J^{-1}\,\mathcal{U}^{\delta}(\xi)\,J\,.
\]
To check the latter equality one simply needs to notice that the right-hand side term solves the same Cauchy problem as $\mathcal{V}^{\delta}$. As a direct consequence,
\begin{align}\label{defdualbase}
    \tilde{q}_1^{\delta}(\xi,\cdot)&:=\f{J^{-1} q_2^{\delta}(\xi,\cdot)}{\langle J^{-1} q_2^{\delta}(\xi,\cdot), q_1^{\delta}(\xi,\cdot)\rangle}\,,&
    \tilde{q}_2^{\delta}(\xi,\cdot)&:=\f{-J^{-1} q_1^{\delta}(\xi,\cdot)}{\langle -J^{-1} q_1^{\delta}(\xi,\cdot), q_2^{\delta}(\xi,\cdot)\rangle}\,,
\end{align}
with
\[
\langle J^{-1} q_j^{\delta}(\xi,\cdot), q_{\ell}^{\delta}(\xi,\cdot)\rangle
\equiv
\langle J^{-1} q_j^{\delta}(0,\cdot), q_{\ell}^{\delta}(0,\cdot)\rangle\,,
\qquad 1\leq j,\ell\leq 2\,.
\]

We have achieved the sought reduction since
\[
\sigma(L_\xi^\delta)\cap B_{\epsilon_0}(0)\,=\,\sigma(D_\xi^\delta)
\]
with
\[
D_{\xi}^{\delta}:=\big(\inp{\tilde{q}_j^{\delta}(\xi,\cdot)}{L_{\xi}^{\delta}q_\ell^{\delta}(\xi,\cdot)}\big)_{1\leq j,\ell\leq 2}\,.
\]

Before expanding $D_\xi^\delta$ in $(\xi,\delta)$ small, we would like to point out a few of its symmetries, inherited from those of $L_\xi^\delta$. From \eqref{defdualbase} and the Hamiltonian structure $L_\xi^\delta=J\,A_\xi^\delta$, with $J$ skew-symmetric and $A_\xi^\delta$ self-adjoint, one readily gets
\[
D_\xi^\delta
\,=\,-\frac{1}{\langle J^{-1} q_2^{\delta}(0,\cdot), q_1^{\delta}(0,\cdot)\rangle}
\begin{pmatrix}
\langle q_2^{\delta}(\xi,\cdot), A_\xi^\delta q_1^{\delta}(\xi,\cdot)\rangle&
\langle q_2^{\delta}(\xi,\cdot), A_\xi^\delta q_2^{\delta}(\xi,\cdot)\rangle\\
-\langle q_1^{\delta}(\xi,\cdot), A_\xi^\delta q_1^{\delta}(\xi,\cdot)\rangle&
-\langle q_1^{\delta}(\xi,\cdot), A_\xi^\delta q_2^{\delta}(\xi,\cdot)\rangle
\end{pmatrix}
\]
with $\langle J^{-1} q_2^{\delta}(0,\cdot), q_1^{\delta}(0,\cdot)\rangle\in\mR$,
\begin{align*}
\langle q_2^{\delta}(\xi,\cdot), A_\xi^\delta q_1^{\delta}(\xi,\cdot)\rangle
&=\overline{\langle q_1^{\delta}(\xi,\cdot), A_\xi^\delta q_2^{\delta}(\xi,\cdot)\rangle}\,,&
\langle q_j^{\delta}(\xi,\cdot), A_\xi^\delta q_j^{\delta}(\xi,\cdot)\rangle
&\in\mR\,,\quad j=1,2\,.
\end{align*}
For the sake of concreteness, let us write
\[
D_\xi^\delta\,=\,\begin{pmatrix}
a(\xi,\delta)&c(\xi,\delta)\\
b(\xi,\delta)&-\overline{a(\xi,\delta)}\end{pmatrix}\,,
\qquad b(\xi,\delta),\,c(\xi,\delta)\in\mR\,.
\]

Moreover from the evenness of $(\uu^\delta,(\ue^\delta)')$, we also have
 \begin{align}\label{prop-L}
L_{\xi}^{\delta}\mathcal{I}&=-\mathcal{I}\,L_{\xi}^{\delta}\,,&
A_{\xi}^{\delta}\mathcal{I}&=\mathcal{I}\,A_{\xi}^{\delta}\,,
 \end{align}
where
\[
(\mathcal{I} g)(x):= \diag (1, -1) \overline{g(-x)}\,.
\]
This implies that
\begin{align*}
\Pi_{\xi}^{\delta}\mathcal{I}&=-\mathcal{I}\,\Pi_{\xi}^{\delta}\,,&
\cU^{\delta}(\xi)\mathcal{I}&=\mathcal{I}\,\cU^{\delta}(\xi)\,.
\end{align*}
Therefore
\begin{align*}
q_1^{\delta}(\xi,\cdot)&\,=\,\mathcal{I}\,q_1^{\delta}(\xi,\cdot)\,,&
q_2^{\delta}(\xi,\cdot)&\,=\,-\mathcal{I}\,q_2^{\delta}(\xi,\cdot)\,,
\end{align*}
since the equality holds when $\xi=0$, thanks to the evenness properties of nearby profiles. By using that
\[
\langle\mathcal{I}f,g\rangle=\overline{\langle f,\mathcal{I}g\rangle}
\]
we deduce that $a(\xi,\delta)=-\overline{a(\xi,\delta)}$, hence
\[
a(\xi,\delta)\in i\mR\,.
\]

\subsection{Further constant-coefficient expansions}

We know $D_0^\delta$ rather explicitly. To expand $D_\xi^\delta$ in $(\xi,\delta)$ small, a significant part of missing information may be derived by expanding, as we do now, $D_\xi^0$ in $\xi$ small.

To begin with, we expand $\mathcal{U}^0(\xi)$.

\begin{prop}\label{prop-extop}
The extension operator $\mathcal{U}^0({\xi})$, defined in \eqref{def-extop}, satisfies
\begin{align}\label{exp-extop}
&\mathcal{U}^0({\xi})(g)
\stackrel{\xi\to0}{=}g
+\alpha\,\xi\,\diag(1,-1)\,\left(e^{-2i\pi \cdot}\langle e^{-2i\pi \cdot},g\rangle
-e^{2i\pi \cdot}\langle e^{2i\pi \cdot},g\rangle\right)
\nonumber\\
&+\xi^2\left(\f{1}{2}\alpha^2\,I+\beta\,\diag (1,-1)\right)
\left(e^{-2i\pi \cdot}\langle e^{-2i\pi \cdot},g\rangle
+e^{2i\pi \cdot}\langle e^{2i\pi \cdot},g\rangle\right)+\cO(\xi^3)\|g\|_{L^2(\mR/\mZ)}\,,
\end{align}
where the constants $\alpha$, $\beta$ are given by
\begin{align}\label{defalp-beta}
\alpha&:=\f{\p_\xi\omega_1(0)}{2\,\omega_1(0)},&
\beta&:=\f{(\p_\xi\omega_1(0))^2-\p_\xi^2\omega_1(0)\,\omega_1(0)}{2\,\omega_1(0)^2}\,,
\end{align}
with $\omega_1$ is as in \eqref{defomega}.
\end{prop}

The proof of the proposition is elementary but a bit long so we moved it to Appendix~\ref{s:ext-xi}. Combined with \eqref{basis-0xi}, it yields
\begin{align*}
k_0\,q_1^0(\xi, x)=\left( \begin{array}{c}
     \cos (2\pi x) \\
          -\omega_1(0) \sin(2\pi x)
    \end{array}\right)
    &+ i\alpha\left( \begin{array}{c}
     -\sin (2\pi x) \\
          \omega_1(0) \cos(2\pi x)
    \end{array}\right)\xi\\
    &+\f{1}{2} \left( \begin{array}{c}
    (\alpha^2+\beta) \cos (2\pi x) \\
          (\beta-\alpha^2)\omega_1(0) \sin(2\pi x)
    \end{array}\right)\xi^2+\cO(\xi^3),
\end{align*}
\begin{align*}
2\pi\,k_0\ q_2^0(\xi, x)=\left( \begin{array}{c}
     \sin (2\pi x) \\
          \omega_1(0) \cos(2\pi x)
    \end{array}\right)
    &+ i\alpha \left( \begin{array}{c}
     \cos(2\pi x) \\
          \omega_1(0) \sin(2\pi x)
    \end{array}\right)\xi\\
    &+ \f{1}{2} \left( \begin{array}{c}
    (\alpha^2+\beta) \sin(2\pi x) \\
          (\alpha^2-\beta)\omega_1(0) \cos(2\pi x)
    \end{array}\right)\xi^2+\cO(\xi^3)\,.
\end{align*}

From a simple examination of the first coordinate, one deduces an expansion for the transition matrix $T(\xi)$
\begin{align}\label{tranofbase}
\begin{pmatrix}
q_1^0(\xi, \cdot)
&q_2^0(\xi,\cdot)\end{pmatrix}
=\begin{pmatrix}
\varphi_{-}^{1}(\xi,\cdot)
&\varphi_{+}^{-1}(\xi,\cdot)
\end{pmatrix}\,T(\xi)
\end{align}
in the form
\begin{align*}
T(\xi)&=
\begin{pmatrix}
(1-\alpha\xi+\frac12(\alpha^2+\beta)\xi^2)&0\\
0&(1+\alpha\xi+\frac12(\alpha^2+\beta)\xi^2)
\end{pmatrix}
\begin{pmatrix}
\frac{1}{2\,k_0}&-\frac{i}{4\,\pi\,k_0}\\
\frac{1}{2\,k_0}&\frac{i}{4\,\pi\,k_0}
\end{pmatrix}
+\cO(\xi^3)\,.
\end{align*}

This is sufficient to derive the sought expansion
\begin{align*}
D_{\xi}^0&= T^{-1}(\xi) \left( \begin{array}{cc}
     \lambda_{-}^{1}(\xi)&0\\
      0& \lambda_{+}^{-1}(\xi)
  \end{array}\right)T(\xi)\\
  &=\f{1}{2}\left( \begin{array}{cc}
     \lambda_{-}^{1}+ \lambda_{+}^{-1}  & -\f{1}{2i\pi}(\lambda_{+}^{-1}- \lambda_{-}^{1})\\
     -{2i\pi}(\lambda_{+}^{-1}- \lambda_{-}^{1} )&\lambda_{-}^{1}+ \lambda_{+}^{-1}
  \end{array}\right)(\xi)+\cO(\xi^3)\\
  &=\left( \begin{array}{cc}
   i( Vk_0-\omega_1'(0))\xi  & -\f{\omega_1''(0)}{4\pi }\xi^2\\[5pt]
     \pi\omega_1''(0)\xi^2 & i( Vk_0-\omega_1'(0))\xi
\end{array}\right)+\cO(\xi^3)\,.
\end{align*}
%MR: in the last formula, off-diagonal differ by a factor 2 from the previous
Let us observe that $\omega_1''(0)>0$ since
\[
\omega_1''(\xi)
\,=\,\frac{(\omega_1'(\xi))^2}{\omega_1(\xi)}+\omega_1(\xi)\,P'(1)k_0^2>0\,.
\]
%\[
%\omega_1''(0)=\f{P'(1)k_0^2}{\omega_1^3(0)}>0.
%\]

\subsection{Expansion of $D_{\xi}^{\delta}$}

We now turn to the full expansion of $D_\xi^\delta$ in $(\xi,\delta)$ small. We recall that
\[
D_\xi^\delta\,=\,a(\xi,\delta)\,I
+\begin{pmatrix}0&c(\xi,\delta)\\b(\xi,\delta)&0\end{pmatrix}\,,\qquad
a(\xi,\delta)\in i\mR\,,\ b(\xi,\delta)\in\mR,\ c(\xi,\delta)\in\mR
\]
so that we want to prove that $b(\xi,\delta)c(\xi,\delta)<0$ when $(\xi,\delta)$ is small.

We already know that for some real $b_1(\delta)$, $c_1(\delta)$,
\begin{align*}
b(\xi,\delta)&\stackrel{(\xi,\delta)\to(0,0)}{=}
b_{1}(\delta) \xi+b_{20}\xi^2+\cO\left(|\xi|^2(|\xi|+\delta)\right)\\
c(\xi,\delta)&\stackrel{(\xi,\delta)\to(0,0)}{=}
c_{1}(\delta)\xi  +c_{20}\xi^2+ c_{02}\delta^2
+ \cO\big(|\xi|^2(|\xi|+\delta)+\delta^3\big)
\end{align*}
with $b_{20}>0$, $c_{20}<0$, $c_{02}<0$. We prove now that
\begin{align*}
b_1({\delta})&\equiv0\,,&
c_1(\delta)&\stackrel{\delta\to0}{=}\cO(\delta^2)\,,
\end{align*}
which is sufficient to conclude the proof.

The vanishing of $b_1$ is an extremely robust property, directly related to the rigorous analysis of slow modulation theory. It follows from computations at the beginning of  Section~\ref{s:reduc-mod}. Indeed since $q_1^\delta(0,\cdot)$ lies in the kernel of $L_0^\delta$ and $\tilde{q}_2^{\delta}(0, \cdot)$ lies in the kernel of $(L_0^\delta)^*$,
\[
b_1(\delta)\,=\,
\langle \tilde{q}_2^{\delta}(0, \cdot), ik^{(\delta)} L_{[1]}^{\delta} q_1^{\delta}(0,\cdot)\rangle\,.
\]
Moreover we already know that
\begin{align*}
L_{[1]}^{\delta}q_1^{\delta}(0,\cdot)
+\frac{k^{(\delta)}}{\p_{V}k^{(\delta)}} q_1^{\delta}(0,\cdot)
\in \Ran(L_0^\delta)
\end{align*}
hence $L_{[1]}^{\delta} q_1^{\delta}(0,\cdot)$ is orthogonal to $\tilde{q}_2^{\delta}(0, \cdot)$. This proves the claim on $b_1$.

Let us now prove that $c_{1}(\delta)=O(\delta^2)$. We already know that $L_0^{\delta}q_2^{\delta}(0,\cdot)=\cO(\delta^2)$ and  $(L_0^{\delta})^{*}\tilde{q}_1^{\delta}(0, \cdot)=\cO(\delta^2)$ so that
\[
c_1(\delta)\,=\,
\langle \tilde{q}_1^{\delta}(0, \cdot), ik^{(\delta)} L_{[1]}^{\delta} q_2^{\delta}(0,\cdot)\rangle+\cO(\delta^2)\,.
\]
Now, since both $c_1(\delta)$ and $\langle \tilde{q}_1^{\delta}(0, \cdot),L_{[1]}^{\delta} q_2^{\delta}(0,\cdot)\rangle$ are real, this implies the claim on $c_1$.

This achieves the proof of Theorem~\ref{th4}.

\begin{rmk}\label{rk:BF}
When $k_2^V<0$, one has $c_{02}>0$ and concludes to instability with eigenvalues of positive real part of size $\cO(\delta)$.
\end{rmk}

\section{Non modulational spectral instability}\label{sec5}

In the present section we prove Theorem~\ref{main} by considering the $\ell=3$ crossings of Lemma~\ref{l:0-lemma}.

We shall also complement Theorem~\ref{main} with the following proposition that proves that %except possibly for two power pressure laws
the instability index $\Gamma$ vanishes at most a finite number of times.

\begin{prop}\label{th5}
In the case $P(\cdot)=T(\cdot)^\gamma$ with $T>0$ and $\gamma\in[1,+\infty[$, the instability index from Theorem~\ref{main} vanishes at most a finite number of times.
\end{prop}

Let us first recall that the two crossings $\ell=3$ crossings of Lemma~\ref{l:0-lemma} are conjugate one to the other through real symmetry. We may thus focus on $(j_{+,3},j'_{+,3},\xi_{+,3})$. For the sake of concision, in the present section, we simply set
\begin{align*}
j&:=j_{+,3}\,,&
j'&:=j_{+,3}'\,,&
\xi_0&:=\xi_{+,3}\,,&
\lambda_0&:=\lambda_{+}^{j'}(\xi_0)=\lambda_{-}^{j}(\xi_0)\,.
\end{align*}

For some values of $V$, it happens that $\xi_0=-\pi$. With our normalization of the Floquet parameter $\xi$ as belonging to $[-\pi,\pi)$, perturbing in $\xi$ then requires to consider both $\xi$ near $-\pi$ and $\xi$ near $\pi$. This introduces extra notational complexity but no significant mathematical difference. Therefore for the sake of simplicity we restrict our proof to the case $\xi_0\neq -\pi$.

\subsection{Finite-dimensional reduction}

Our proof of Theorem~\ref{main} follows very closely its sketch in the introduction. It also shares many similarities with the proof of Theorem~\ref{th4}.

The main strategical difference is that we build our basis by extending in $\delta$ an explicit choice made for $\delta=0$. An obvious reason for this discrepancy is that here the $\xi=\xi_0$ analysis has no evident structure when $\delta>0$. The main departure in computations is that here we need expansions that are higher-order with respect to $\delta$.

To begin with, we thus set
\begin{align*}
q_{+}^{0}(\xi,\cdot)&:=\varphi_{+}^{j'}(\xi,\cdot)\,,&
q_{-}^{0}(\xi,\cdot)&:=\varphi_{-}^{j}(\xi,\cdot)\,,
\end{align*}
with $\varphi_{\pm}^{l}$ being defined in \eqref{defbasis-constant}.
Then we extend those through
\begin{align*}
q_{+}^{\delta}(\xi,\cdot)&:=\cU^\xi(\delta)q_{+}^{0}(\xi,\cdot)\,,&
q_{-}^{\delta}(\xi,\cdot)&:=\cU^\xi(\delta)q_{+}^{0}(\xi,\cdot)\,.
\end{align*}
where the extension operator $\cU^\xi(\delta)$ is obtained by solving
\begin{align}\label{def-extop-delta}
\p_\delta\mathcal{U}^\xi(\delta)
&=[\p_\delta\Pi_{\xi}^{\delta},\Pi_{\xi}^{\delta}]\ \mathcal{U}^\xi(\delta)\,,
&\mathcal{U}^{\xi}(0)=I\, ,
\end{align}
from the spectral projector
\[
\Pi_{\xi}^{\delta}:=\f{1}{2i \pi} \oint_{\p B_{\epsilon_1}(\lambda_0)}
(\lambda I-L_{\xi}^{\delta})^{-1}\, \d \lambda
\]
defined with $\epsilon_1>0$ sufficiently small, $0\leq \delta \leq\epsilon_1$, $|\xi-\xi_0|\leq \epsilon_1$.

Let us also determine a dual basis spanning the sum of characteristic spaces of $(L_\xi^\delta)^*$ associated with eigenvalues in $B_{\epsilon_1}(\overline{\lambda_0})$. Since $j\neq j'$,
\begin{align*}
\langle J^{-1}\,q_{+}^{0}(\xi,\cdot),q_{-}^{0}(\xi,\cdot)\rangle&=0\,,&
\langle J^{-1}\,q_{-}^{0}(\xi,\cdot),q_{+}^{0}(\xi,\cdot)\rangle&=0\,,&
\end{align*}
whereas direct computations give
\begin{align}\label{Krein}
\langle J^{-1}q_{+}^{0}(\xi,\cdot), q_{+}^{0}(\xi,\cdot)\rangle&=2i\omega_{j'}(\xi)\,,&
\langle J^{-1}q_{-}^{0}(\xi,\cdot), q_{-}^{0}(\xi,\cdot)\rangle&=-2 i  \omega_{j}(\xi)\,.
\end{align}
Therefore
\begin{align*}
\displaystyle
\tilde{q}_{+}^{0}(\xi, \cdot)&=\f{i\,J^{-1}q_{+}^{0}(\xi, \cdot)}{2\omega_{j'}(\xi)}\,,&
\tilde{q}_{-}^{0}(\xi, \cdot)=-\f{i\,J^{-1}q_{-}^{0}(\xi, \cdot)}{2\omega_{j}(\xi)}.
\end{align*}
Then one may extend these formula with an extension operator associated with $(L_\xi^\delta)^*$ and use Hamiltonian symmetry to check that still
\begin{align*}
\displaystyle
\tilde{q}_{+}^{\delta}(\xi, \cdot)&=\f{i\,J^{-1}q_{+}^{\delta}(\xi, \cdot)}{2\omega_{j'}(\xi)}\,,&
\tilde{q}_{-}^{\delta}(\xi, \cdot)=-\f{i\,J^{-1}q_{-}^{\delta}(\xi, \cdot)}{2\omega_{j}(\xi)}.
\end{align*}

With this in hands, we may set
\begin{align*}
 D_{\xi}^{\delta}= \left( \begin{array}{cc}
\langle \tilde{q}_{+}^{\delta}(\xi, \cdot), L_{\xi}^{\delta}  {q}_{+}^{\delta}(\xi, \cdot)\rangle & \langle \tilde{q}_{+}^{\delta}(\xi, \cdot), L_{\xi}^{\delta}  {q}_{-}^{\delta}(\xi, \cdot)\rangle\\[3pt]
 \langle \tilde{q}_{-}^{\delta}(\xi, \cdot), L_{\xi}^{\delta}  {q}_{+}^{\delta}(\xi, \cdot)\rangle & \langle \tilde{q}_{-}^{\delta}(\xi, \cdot), L_{\xi}^{\delta}  q_{-}^{\delta}(\xi, \cdot)\rangle
    \end{array}\right).
\end{align*}
so that
\[
\sigma(L_\xi^\delta)\cap B_{\epsilon_1}(\lambda_0)=\sigma(D_\xi^\delta),
\]
when $0\leq \delta \leq\epsilon_1$, $|\xi-\xi_0|\leq \epsilon_1$.

\begin{rmk}\label{rk:Krein}
Our present normalization differs from the one used in the discussion of the introduction, by the fact that that we do not scale $(q_{+}^{\delta}(\xi, \cdot),q_{-}^{\delta}(\xi, \cdot))$ to force symmetry between the direct and dual problems. A basis as in the introduction is obtained by setting
\begin{align*}
 q_{1}^{\delta}(\xi, \cdot)&=\frac{q_{+}^{\delta}(\xi, \cdot)}{\sqrt{2\omega_{j'}(\xi)}}\,,&
q_{2}^{\delta}(\xi, \cdot)&=\frac{q_{-}^{\delta}(\xi, \cdot)}{\sqrt{2\omega_{j}(\xi)}}\,,
\end{align*}
so that
\begin{align*}
\tilde{q}_{1}^{\delta}(\xi, \cdot)&=iJ^{-1}{q}_{1}^{\delta}(\xi, \cdot)\,,&
\tilde{q}_{2}^{\delta}(\xi, \cdot)&=-iJ^{-1}{q}_{2}^{\delta}(\xi, \cdot)\,.
\end{align*}
Consistently the index $\Gamma$ derived here differs from the one of the introductory exposition by an immaterial nonzero scaling factor. Let us also point out that it is \eqref{Krein} that shows that the two eigenvalues colliding at $\xi=\xi_0$ have opposite Krein signature, which prevents to enforce either $(\tilde{q}_{1}^{\delta}(\xi, \cdot),\tilde{q}_{2}^{\delta}(\xi, \cdot))=(iJ^{-1}{q}_{1}^{\delta}(\xi, \cdot),iJ^{-1}{q}_{2}^{\delta}(\xi, \cdot))$ or $(\tilde{q}_{1}^{\delta}(\xi, \cdot),\tilde{q}_{2}^{\delta}(\xi, \cdot))=(-iJ^{-1}{q}_{1}^{\delta}(\xi, \cdot),-iJ^{-1}{q}_{2}^{\delta}(\xi, \cdot))$, which would lead to
\begin{align*}
\textrm{either }&&
 D_{\xi}^{\delta}&= i\begin{pmatrix}
\langle q_{1}^{\delta}(\xi, \cdot), A_{\xi}^{\delta}  q_{1}^{\delta}(\xi, \cdot)\rangle & \langle q_{1}^{\delta}(\xi, \cdot), A_{\xi}^{\delta}  q_{2}^{\delta}(\xi, \cdot)\rangle\\[3pt]
 \langle q_{2}^{\delta}(\xi, \cdot), A_{\xi}^{\delta}  q_{1}^{\delta}(\xi, \cdot)\rangle & \langle q_{2}^{\delta}(\xi, \cdot), A_{\xi}^{\delta}  q_{2}^{\delta}(\xi, \cdot)\rangle
    \end{pmatrix}\,,\\
    \textrm{or }&&
     D_{\xi}^{\delta}&= -i\begin{pmatrix}
\langle q_{1}^{\delta}(\xi, \cdot), A_{\xi}^{\delta}  q_{1}^{\delta}(\xi, \cdot)\rangle & \langle q_{1}^{\delta}(\xi, \cdot), A_{\xi}^{\delta}  q_{2}^{\delta}(\xi, \cdot)\rangle\\[3pt]
 \langle q_{2}^{\delta}(\xi, \cdot), A_{\xi}^{\delta}  q_{1}^{\delta}(\xi, \cdot)\rangle & \langle q_{2}^{\delta}(\xi, \cdot), A_{\xi}^{\delta}  q_{2}^{\delta}(\xi, \cdot)\rangle
    \end{pmatrix}\,.
\end{align*}
We stress that the computation \eqref{Krein} holds for any of the crossings in Lemma~\ref{l:0-lemma}.
\end{rmk}

By using the Hamiltonian structure, one gets that $D_\xi^\delta$ takes the following form
\[
D_{\xi}^{\delta}
=\frac{i}{2}\begin{pmatrix}
\dfrac{b_+(\xi,\delta)}{\omega_{j'}(\xi)}&
\dfrac{a(\xi,\delta)}{\omega_{j'}(\xi)}\\[1em]
-\dfrac{\overline{a(\xi,\delta)}}{\omega_{j}(\xi)}&
-\dfrac{b_-(\xi,\delta)}{\omega_{j}(\xi)}
\end{pmatrix}
\]
with
\begin{align*}
b_+(\xi,\delta)&=\langle q_{+}^{\delta}(\xi, \cdot),
A_{\xi}^{\delta}  q_{+}^{\delta}(\xi, \cdot)\rangle\in\mR\,,&
b_-(\xi,\delta)&=\langle q_{-}^{\delta}(\xi, \cdot),
A_{\xi}^{\delta}  q_{-}^{\delta}(\xi, \cdot)\rangle\in\mR\,,\\
\end{align*}
and
\begin{align*}
a(\xi,\delta)&=\langle q_{+}^{\delta}(\xi, \cdot),
A_{\xi}^{\delta}  q_{-}^{\delta}(\xi, \cdot)\rangle
\,=\,\overline{\langle q_{-}^{\delta}(\xi, \cdot),
A_{\xi}^{\delta}  q_{+}^{\delta}(\xi, \cdot)\rangle}\,.
\end{align*}
Eigenvalues of $D_\xi^\delta$ are given by
\[
\frac{i}{2}
\left(\frac12\left(\frac{b_+(\xi,\delta)}{\omega_{j'}(\xi)}
-\frac{b_-(\xi,\delta)}{\omega_{j}(\xi)}\right)
\pm\sqrt{
\frac14\left(\frac{b_+(\xi,\delta)}{\omega_{j'}(\xi)}
+\frac{b_-(\xi,\delta)}{\omega_{j}(\xi)}\right)^2
-\frac{|a(\xi,\delta)|^2}{\omega_{j'}(\xi)\omega_{j}(\xi)}
}\right)\,.
\]

Note that moreover by design,
\[
D_\xi^0\,=\,
\begin{pmatrix}
\lambda_+^{j'}(\xi)&0\\
0&\lambda_-^{j}(\xi)
\end{pmatrix}\,.
\]
In particular
\begin{align*}
\p_\xi\left(\frac{b_+(\cdot,0)}{2\omega_{j'}}+\frac{b_-(\cdot,0)}{2\omega_{j}}\right)(\xi_0)
&\,=\,\frac{1}{i}\,\p_\xi(\lambda_+^{j'}-\lambda_-^j)(\xi_0)
\,=\,\p_\xi(\omega_{j'}+\omega_{j})(\xi_0)>0
%&\,=\,P'(1)\,k_0^2\,2\pi\left(\frac{j'+\frac{\xi_0}{2\pi}}{\omega_{j'}^{3}(\xi_0)}+\frac{j+\frac{\xi_0}{2\pi}}{\omega_{j}^{3}(\xi_0)}\right)
\end{align*}
where the sign information stems from the fact that $\zeta\mapsto \sqrt{1+P'(1)\,\zeta^2}$ is even and strictly convex and $|j'+\xi_0/(2\pi)|<j+\xi_0/(2\pi)$ since $j'=j-3$ and $j+\xi_0/2\pi>3/2$. Therefore there exists a smooth $\Xi$ defined on a neighborhood of $0$ in $\mR^+$ such that $\Xi(0)=\xi_0$ and, for any small $\delta$,
\[
\left(\frac{b_+(\cdot,\delta)}{2\omega_{j'}}+\frac{b_-(\cdot,\delta)}{2\omega_{j}}\right)(\Xi(\delta))
\,=\,0\,.
\]
To conclude the proof it is thus sufficient to find a $\Gamma$ such that
\begin{align}\label{eq:main-aux}
a(\Xi(\delta),\delta)
\,\stackrel{\delta\to 0}{=}\,\Gamma\,\delta^3+\cO(\delta^4)\,.
\end{align}

\begin{rmk}\label{rk:ell}
The above argument may be extended to any of the crossings of Lemma~\ref{l:0-lemma} with $\ell\geq3$ since $j_{+,\ell}+\xi_{+,\ell}/(2\pi)>\ell/2$. It fails however for the crossing at the origin since $j_{2}+\xi_{2}/(2\pi)=1$. Since the remaining part of the argument applies to any crossing, a statement similar to Theorem~\ref{main} may be obtained for any crossing far from the origin. In contrast, as seen in Section 4, the analysis of the $\ell=2$ crossing required a different analysis, with instability decided not by a non-vanishing condition (which is easily satisfied) but by a sign condition, here by the sign of $k_2^V$.
\end{rmk}

The sought \eqref{eq:main-aux} is derived from the fact that for any $m\in\mN$
\[
\p_\delta^m(\ue,\uu)^0\in\Span\left(\left\{\,e^{2\pi\,p\,(\cdot)\,}\,X\,;\,|p|\leq m\,,\,
X\in\mC^2\,\right\}\right)
\]
and $j-j'=3$. To be more concrete, we introduce notation $M^m_p:\mC^2\to\mC^2$ defined by
\[
M^m_p\,(X):=\langle e^{2\pi\,(m+p)\,(\cdot)};M(e^{2\pi\,m\,(\cdot)}\,X)\rangle
\]
for any operator $M$ on $L^2(\mR/\mZ)^2$.  The above piece of information on profiles readily yields
\begin{align*}
(\p_\delta^m L_\xi^0)^p_r&\equiv 0\,,&
(\p_\delta^m \Pi_\xi^0)^p_r&\equiv 0\,,&
\textrm{when }|r|>m\,.
\end{align*}
Then, by expanding in $\delta$, for $\sigma\in\{+,-\}$,
\begin{align*}
q_\sigma^\delta(\xi,\cdot)&\,=\,\Pi_\xi^\delta(q_\sigma^\delta(\xi,\cdot))\,,&
%\tilde{q}_\sigma^\delta(\xi,\cdot)&\,=\,(\Pi_\xi^\delta)^*(\tilde{q}_\sigma^\delta(\xi,\cdot))\,,&
L_\xi^\delta q_\sigma^\delta(\xi,\cdot)&\,=\,\Pi_\xi^\delta(L_\xi^\delta q_\sigma^\delta(\xi,\cdot))\,,&
\end{align*}
we derive recursively that for any $\xi$ and any $m$
\begin{align*}
\p_\delta ^m q_+^0(\xi,\cdot),\
\p_\delta^m(L_\xi^\delta q_+^\delta(\xi,\cdot))_{|\delta=0}
&\in \Span\left(\left\{\,e^{2\pi\,(j'+p)\,(\cdot)\,}\,X\,;\,|p|\leq m\,,\,
X\in\mC^2\,\right\}\right)\,,\\
\p_\delta ^m q_-^0(\xi,\cdot),\
\p_\delta^m(L_\xi^\delta q_-^\delta(\xi,\cdot))_{|\delta=0}
&\in \Span\left(\left\{\,e^{2\pi\,(j+p)\,(\cdot)\,}\,X\,;\,|p|\leq m\,,\,
X\in\mC^2\,\right\}\right)\,.
\end{align*}
Since $j=j'-3$, by orthogonality of trigonometric monomials, this implies that for any $\xi$,
\begin{align*}
\p_\delta^m a(\xi,0)&\,=\,0\,,&
\textrm{when }m\in\{0,1,2\}\,.
\end{align*}
This concludes the proof of \eqref{eq:main-aux} with $\Gamma:=\p_\delta^3 a(\xi_0,0)/6$ (since $\Xi(\delta)=\xi_0+\cO(\delta)$).

It also concludes the proof of Theorem~\ref{main}. Yet the value of Theorem~\ref{main} hinges on the fact that most of the time $\Gamma$ is nonzero. To support this claim, we provide a relatively explicit formula for $\Gamma$ and use it to prove Proposition~\ref{th5}.

\subsection{Computation of $\Gamma$}

The computation of $\Gamma$ is quite cumbersome so that even small notational gains are helpful. For this reason, we set
\begin{align*}
L^{[m]}&:=\frac{1}{m!}\p_\delta^m L_{\xi_0}^0\,,&
\cU^{[m]}&:=\frac{1}{m!}\p_\delta^m \cU^{\xi_0}(0)\,,&
%\kappa_0&:=2\pi k_0\,,%\quad
%\kappa_2:=2\pi k_2\,,&
%\\
V_+&:=\begin{pmatrix}1\\-\,i\omega_{j'}(\xi_0)\end{pmatrix}\,,&
V_-&:=\begin{pmatrix}1\\i\omega_{j}(\xi_0)\end{pmatrix}\,,%&
%s&:=j+\frac{\xi_0}{2\pi}-\frac32=j'+\frac{\xi_0}{2\pi}-\frac32\,,
\end{align*}
\begin{align*}
    \cM_{+}&:=  \cM_{+}^{j'}(\xi_0)
    \,=\,\f{1}{2}\left(\begin{array}{cc}
       1 & -\f{1}{i\,\omega_{j'}(\xi_0)} \\[5pt]
       - {i\,\omega_{j'}(\xi_0)}  &  1
    \end{array}\right)\,,\\
        \cM_{-}&:=  \cM_{-}^{j}(\xi_0)
    \,=\,\f{1}{2}\left(\begin{array}{cc}
       1 & \f{1}{i\,\omega_j(\xi_0)} \\[5pt]
       i\,\omega_j(\xi_0)  &  1
    \end{array}\right)\,,
\end{align*}
and, for $m\notin\{j,j'\}$,
%$\lambda\notin \{\lambda_+^m(\xi_0),\lambda_-^m(\xi_0)\}$
\begin{align*}
G_m&:=(\lambda_0\,I-\Sigma^L_0(k_0\,(2\pi\,m+\xi_0)))^{-1}\,\,.
\end{align*}

With this in hands, a first expression of $\Gamma$ is
\begin{align*}
\Gamma&=\big\langle J(\cU^{[3]})_{3}^{j'}V_{+}, (L^{[0]})_0^j V_{-}^j\big\rangle
+\big\langle J (\cU^{[2]})_{2}^{j'}V_{+},  \big((L^{[0]})_0^{j-1} (\cU^{[1]})_{-1}^{j}+(L^{[1]})_{-1}^{j} \big)V_{-}\big\rangle,\\
&\quad+ \big\langle J (\cU^{[1]})_{1}^{j'}V_{+},  \big((L^{[0]})_0^{j-2} (\cU^{[2]})_{-2}^{j}+ (L^{[1]})_{-1}^{j-1}(\cU^{[1]})_{-1}^{j}+  (L^{[2]})_{-2}^{j} \big)V_{-}\big\rangle, \\
&\quad+\big\langle J V_{+},  \big((L^{[0]})_0^{j-3} (\cU^{[3]})_{-3}^{j}+ (L^{[1]})_{-1}^{j-2}(\cU^{[2]})_{-2}^{j}+  (L^{[2]})_{-2}^{j-1}(\cU^{[1]})_{-1}^{j}+ (L^{[3]})_{-3}^{j} \big)V_{-}\big\rangle.
\end{align*}
with $\langle\,\cdot\,,\,\cdot\,\rangle$ denoting the canonical scalar product on $\mC^2$ skew-linear in its first component.

The first stage is to obtain explicit expressions for $(\cU^{[m]})^{j}_{-m}$ and $(\cU^{[m]})^{j'}_{m}$, $m\in\{1,2,3\}$.

\begin{prop}\label{prop:ext-delta}
One has
\begin{align*}
(\cU^{[1]})_{-1}^j&=G_{j-1}(L^{[1]})_{-1}^{j}\cM_{-}\,,&
(\cU^{[2]})_{-2}^j&=B_{-}\cM_{-}\,,\\
(\cU^{[1]})_{1}^{j'}&=G_{j'+1}(L^{[1]})_{1}^{j'}\cM_{+}\,,&
(\cU^{[2]})_{2}^{j'}&=B_{+}\cM_{+}\,,
\end{align*}
with
\begin{align*}
B_{-}&:= G_{j-2}(L^{[2]})_{-2}^j
+G_{j-2}(L^{[1]})_{-1}^{j-1}G_{j-1}(L^{[1]})_{-1}^{j}, \\
B_{+}&:= G_{j'+2} (L^{[2]})_{2}^{j'}
+G_{j'+2}(L^{[1]})_{1}^{j'+1}G_{j'+1}(L^{[1]})_{1}^{j'},\\
\end{align*}
and
\begin{align*}
\big(\cU^{[3]}-\f{1}{6}\big[\p_\delta^3\Pi_{\xi_0}^0, \Pi_{\xi_0}^0\big]\big)_{-3}^j
&=-\f{1}{3}\cM_{+} C_{-}G_{j-1}(L^{[1]})_{-1}^{j}\cM_{-}
-\f{2}{3}\cM_{+}(L^{[1]})_{-1}^{j-2}G_{j-2} B_{-}\cM_{-}\,,\\
\big(\cU^{[3]}-\f{1}{6}\big[\p_\delta^3\Pi_{\xi_0}^0, \Pi_{\xi_0}^0\big]\big)_{+3}^{j'}
&= -\f{1}{3} \cM_{-} C_{+}G_{j'+1}(L^{[1]})_{1}^{j'}\cM_{+}
-\f{2}{3}  \cM_{-} (L^{[1]})_{1}^{j'+2} G_{j'+2} B_{+}\cM_{+}\,,
\end{align*}
where
\begin{align*}
C_{-}
&=(L^{[2]})_{-2}^{j-1}G_{j-1}
+(L^{[1]})_{-1}^{j-2} G_{j-2}(L^{[1]})_{-1}^{j-1}G_{j-1}\,,\\
C_{+}
&=(L^{[2]})_{2}^{j'+1}G_{j'+1}
+ (L^{[1]})_{1}^{j'+2} G_{j'+2}(L^{[1]})_{1}^{j'+1} G_{j'+1}\,.
\end{align*}
\end{prop}

The proof of the proposition is postponed to Appendix~\ref{s:ext-delta}.

Now that we have a quite explicit formula for $\Gamma$, we need to track cancellations to reduce its complexity. Our final result is as follows.

\begin{prop}\label{prop-rmc}
The instability index $\Gamma$ takes the form
\begin{align*}
\Gamma = \big\langle J V_{+}, \mathcal{N} V_{-}\big \rangle
\end{align*}
where
\begin{align*}%\label{defcN}
\cN&:=
(L^{[3]})_{-3}^j
+(L^{[1]})_{-1}^{j-2}G_{j-2} (L^{[2]})_{-2}^j
+(L^{[2]})_{-2}^{j-1}G_{j-1} (L^{[1]})_{-1}^j %\nonumber\\
%&\quad
+(L^{[1]})_{-1}^{j-2}G_{j-2}(L^{[1]})_{-1}^{j-1}G_{j-1}(L^{[1]})_{-1}^{j}\,.
\end{align*}
\end{prop}

The quite spectacular simplification hinges on properties inherited from Hamiltonian symmetry.
%real and Hamiltonian symmetries and evenness/oddness of profiles.
Let us recall that the Hamiltonian symmetry is precisely that $J\,(L_\xi^\delta)^*\,J^{-1}=-L_\xi^\delta$, and observe that $(M^m_p)^*=(M^*)^{m+p}_{-p}$. This provides various relations whose statement is facilitated by the introduction of the piece of notation
\[
S[B]\,:=\,J\,B^*\,J^{-1}
\]
for any matrix $B$. Indeed, all operators we handle are obtained as functions of $L_\xi^\delta$ and from Hamiltonian symmetry stem
\begin{align*}
S[(L^{[m]})^p_q]&=-(L^{[m]})^{p+q}_{-q}\,,&
S[(\p_\delta^m\Pi_{\xi_0}^0)^p_q]&=-(\p_\delta^m\Pi_{\xi_0}^0)^{p+q}_{-q}\,,&
\end{align*}
and
\begin{align*}
S[\cM_{\pm}]&=\cM_{\pm}\,,&
S[G_{m}]&=-G_{m}\,,&
S[B_{\pm}]&=C_{\mp}\,.%,&
\end{align*}
%and thus
%\begin{align*}
%S\left[(\cU^{[1]})_{1}^{j'}\right]&=(\cM_+)^2(L^{[1]})_{-1}^{j-2} G_{j-2}\,,&
%S\left[(\cU^{[1]})_{-1}^j\right]&=(\cM_-)^2(L^{[1]})_{1}^{j-1}G_{j-1}\,.
%\end{align*}

The last piece of computation we need is
\begin{align*}
%(\cM_{\pm})^2&=\cM_{\pm}\,,&
\cM_{\pm}\,V_{\pm}&=V_{\pm}\,.%,
\end{align*}
%that stem from the fact $\cM_{\pm}$ are related to spectral projectors of $\Sigma^L_0(i\,k_0(2\pi\,j+\xi_0))$ and $\Sigma^L_0(i\,k_0(2\pi\,j'+\xi_0))$.

With this in hands, we may prove Proposition~\ref{prop-rmc}.

\begin{proof}[Proof of Proposition~\ref{prop-rmc}]

Let us split $\Gamma$ as
\begin{align*}
\Gamma= \Gamma_0+\Gamma_1+\Gamma_2,
\end{align*}
where
\begin{align*}
\Gamma_0&=\big\langle J V_{+}, (L^{[3]})_{-3}^{j}V_{-}\big\rangle
+\big\langle J (\cU^{[2]})_{2}^{j'}V_{+}, (L^{[1]})_{-1}^{j} V_{-}\big\rangle
+\big\langle J (\cU^{[1]})_{1}^{j'}V_{+},(L^{[2]})_{-2}^{j} V_{-}\big\rangle\,,\\
\Gamma_1&=\big\langle J(\cU^{[3]})_{3}^{j'}V_{+}, (L^{[0]})^j V_{-}\big\rangle
+ \big\langle J V_{+}, (L^{[0]})_0^{j'} (\cU^{[3]})_{-3}^{j}V_{-}\big\rangle, \\
\Gamma_2&=\big\langle J (\cU^{[1]})_{1}^{j'}V_{+},
\big((L^{[0]})_0^{j-2} (\cU^{[2]})_{-2}^{j}
+(L^{[1]})_{-1}^{j-1}(\cU^{[1]})_{-1}^{j}\big)V_{-}\big\rangle\\
 &\quad
 +\big\langle J (\cU^{[2]})_{2}^{j'}V_{+}, (L^{[0]})_0^{j-1} (\cU^{[1]})_{-1}^{j} V_{-}\big\rangle\\
&\quad
+\big\langle J V_{+},
\big((L^{[1]})_{-1}^{j-2}(\cU^{[2]})_{-2}^{j}
+(L^{[2]})_{-2}^{j-1}(\cU^{[1]})_{-1}^{j}\big) V_{-}\big\rangle\,.
\end{align*}

Our goal is to show that $\Gamma_0= \big\langle J V_{+}, \mathcal{N} V_{-}\big \rangle$ and $\Gamma_2=-\Gamma_1$. The first point is straightforward. Note, for instance, that
\begin{align*}
\big\langle J (\cU^{[1]})_{1}^{j'}V_{+},(L^{[2]})_{-2}^{j} V_{-}\big\rangle
&=\big\langle J V_{+},
S[(\cU^{[1]})_{1}^{j'}](L^{[2]})_{-2}^{j} V_{-}\big\rangle\\
%&=\big\langle J V_{+},
%(\cM_+)^2(L^{[1]})_{-1}^{j-2}G_{j-2} (L^{[2]})_{-2}^j V_{-}\big\rangle\\
&=\big\langle J V_{+},
\cM_+(L^{[1]})_{-1}^{j-2}G_{j-2} (L^{[2]})_{-2}^j V_{-}\big\rangle\\
&=\big\langle J S[\cM_+]V_{+},
(L^{[1]})_{-1}^{j-2}G_{j-2} (L^{[2]})_{-2}^j V_{-}\big\rangle\\
&=\big\langle J \cM_+V_{+},
(L^{[1]})_{-1}^{j-2}G_{j-2} (L^{[2]})_{-2}^j V_{-}\big\rangle\\
&=\big\langle J V_{+},
(L^{[1]})_{-1}^{j-2}G_{j-2} (L^{[2]})_{-2}^j V_{-}\big\rangle\,.
\end{align*}

As an intermediate step towards $\Gamma_2=-\Gamma_1$, we observe that
\[
\Gamma_1
\,=\,
- \lambda_0\bigg\langle J V_{+},  \big(C_{-} G_{j-1}(L^{[1]})_{-1}^{j}
+(L^{[1]})_{-1}^{j-2} G_{j-2} B_{-}\big) V_{-}\bigg\rangle
\]
since
\begin{align*}
(L^{[0]})^j_0\,V_-&=\lambda_0\,V_-\,,&
(L^{[0]})^{j'}_0\,V_+&=\lambda_0\,V_+\,,
\end{align*}
and
\begin{align*}
\langle J \big(\big[\p_\delta^3\Pi_{\xi_0}^0, \Pi_{\xi_0}^0\big]\big)_{3}^{j'}\,V_+,V_-\rangle
&=
\langle J V_+,S\left[\big(\big[\p_\delta^3\Pi_{\xi_0}^0, \Pi_{\xi_0}^0\big]\big)_{3}^{j'}\right]\,V_-\rangle\\
&=-\langle J V_+,\big(\big[\p_\delta^3\Pi_{\xi_0}^0, \Pi_{\xi_0}^0\big]\big)_{-3}^j\,V_-\rangle\,.
\end{align*}

Furthermore,
\begin{align*}
\Gamma_2&
=\big\langle J V_{+},\left((L^{[1]})_{-1}^{j-2}\big(I+G_{j-2}(L^{[0]})_0^{j-2}\big)B_{-} %V_{-}\big\rangle\\
%&\quad
+%\big\langle J V_{+},
C_{-} (G_{j-1})^{-1} \big(I+G_{j-1}(L^{[0]})_0^{j-1}\big)(\cU^{[1]})_{-1}^{j}
\right)V_{-}\big\rangle\\
&=-\Gamma_1\,,
\end{align*}
since
\begin{align*}
I+G_{j-2}(L^{[0]})_0^{j-2}&=\lambda_0 G_{j-2}\,,&
I+G_{j-1}(L^{[0]})_0^{j-1}&=\lambda_0 G_{j-1}\,.
 \end{align*}
\end{proof}

 \subsection{Proof of Proposition~\ref{th5}}

Proposition~\ref{th5} follows from the following lemma, whose computationally demanding proof is given in Appendix~\ref{s:computations}.

  \begin{lem}\label{dl-c-s}
  For a pressure given by $P(\rho)=T\rho^\gamma$, $T>0$, $\gamma\geq1$, one has
  \begin{align*}%\label{dl-c}
  \displaystyle
  \Gamma\stackrel{V\to\infty}{=}f(\gamma)+\cO(V^{-1})\quad\textrm{where }
  f(\gamma):=-\f{1}{6144}(65\gamma^3+315\gamma^2+115\gamma-135)\,.
%  f(\gamma)=\f{1}{6144}(-65\gamma^3-315\gamma^2+47405\gamma-256473).
  \end{align*}
  \end{lem}

Let us show how to deduce the proposition from the lemma.

\begin{proof}[Proof of Proposition~\ref{th5}]
Since $f(1)<0$, $f'(1)<0$ and $f''(1)<0$, $f$ does not vanish on $[1,\infty)$. Let us now fix $T>0$ and $\gamma\geq1$, and achieve the proof.

The index $\Gamma$ is analytic on $(\sqrt{P'(1)},+\infty)$ and does not vanish near $\infty$. It remains to show that $\Gamma$ does not vanish near $\sqrt{P'(1)}$. This follows from the fact that $\Gamma$ is non zero and is given as a meromorphic function of $\sqrt{V^2-P'(1)}$.
  \end{proof}

\appendix

\section{Proof of Lemma~\ref{l:0-lemma}}\label{s:0}

Let us first observe by monotonicity that $\lambda_-^{j}(\xi)=\lambda_+^{j'}(\xi)$ implies $j'<j$, and we assume the latter from now on. Setting $\tau=P'(1)/V^2$, $z=j+\xi/2\pi$ and $z'=j'+\xi/2\pi$, equation $\lambda_-^{j}(\xi)=\lambda_+^{j'}(\xi)$ is equivalent to
\[
z-z'\,=\,\sqrt{1-\tau+\tau\,z^2}+\sqrt{1-\tau+\tau\,(z')^2}\,.
\]
Since $z\neq z'$, this implies
\[
\sqrt{1-\tau+\tau\,z^2}-\sqrt{1-\tau+\tau\,(z')^2}\,=\,\tau\,(z+z')\,.
\]
Combining with the original form leaves
\begin{equation}\label{aux}
(\tau-1)\,z'=-(\tau+1)\,z+2\sqrt{1-\tau+\tau\,z^2}\,.
\end{equation}
By tracking manipulations carried out so far one can recover the original equation under the apparent extra assumption that
\[
\frac{\sqrt{1-\tau+\tau\,z^2}-\sqrt{1-\tau+\tau\,(z')^2}}{z-z'}\neq 1\,.
\]
Yet it follows from $\tau<1$ that the left-hand side of the inequality has absolute value less than $1$. Thus we have obtained that  $\lambda_-^{j}(\xi)=\lambda_+^{j'}(\xi)$ is equivalent to \eqref{aux}, also written as
\[
\frac12(1-\tau)\,(z-z')=-\tau\,z+\sqrt{1-\tau+\tau\,z^2}\,.
\]
The latter is equivalent to
\begin{align*}
\frac12(1-\tau)\,(z-z')+\tau\,z&\geq0\,,\\
\frac14\,(z-z')^2-1&=\tau\,(z-\tfrac12(z-z'))^2\,.
\end{align*}
This implies that, as claimed, $j-j'\geq 2$ is necessary. Fixing $\ell=z-z'=j-j'\geq 2$, one may solve the equation for $z$
\[
z\,=\,\tfrac12\ell\pm\sqrt{\frac{\ell^2-4}{4\,\tau}}
\]
and the sign constraint is automatically satisfied (since $\tau<1$).

\section{A glimpse at slow modulation theory}\label{secW}

In the present section, we connect the analysis of Theorem~\ref{th4} with spectral validations of formally derived modulation systems, also known as Whitham systems, and their small-amplitude limit, as studied for large classes of Hamiltonian systems in \cite{JNS-BNR,Nonlinearity-BMR2,SMF-Audiard-Rodrigues}.

\subsection*{Formal derivation}

We first recall how to derive a modulation system from geometrical optics arguments.

To begin with it is useful to gather conservation laws associated with \eqref{eep-int}. From the fact that System~\eqref{eep-int} has a Hamiltonian structure whose Hamiltonian density commutes with spatial translation one deduces that it implies a conservation law for the momentum density $\mathcal{M}$,
\[
\mathcal{M}[E,u]:=-\,u\,\p_xE\,,
\]
generating the group of spatial translations. Namely, applying the abstract computations from \cite[Appendix~A]{SMF-Audiard-Rodrigues}, from \eqref{eep-int} we derive for $U=(E,u)$
\begin{align}\label{eq:cons}
\p_t(\mathcal{M}[U])\,=\,\p_x\left(\mathcal{S}[U]\right)
\end{align}
with
\[
\mathcal{S}[U]:=\nabla_{U_x}\mathcal{M}[U]\cdot J\delta\mathcal{H}[U]-\mathcal{H}[U]+\nabla_{U_x}\mathcal{H}[U]\cdot U_x\,.
\]
%or more explicitly
%\begin{align*}
%\p_t(\mathcal{M}[U])\,=\,\p_x\left(\frac12\,u^2-F(1+\p_x E)-\frac12\,E^2
%+(u^2+F'(1+\p_x E))\p_xE\right)
%\end{align*}

We are now in position to derive a modulation system. To do so, we introduce a one-phase slow/fastly-oscillatory \emph{ansatz}
\begin{align}\label{eq:ansatz}
U^{(\epsilon)}(t,x)
\,=\,U_{(\epsilon)}\left(\epsilon\,t,\epsilon\,x;
\frac{\varphi_{(\epsilon)}(\epsilon\,t,\epsilon\,x)}{\epsilon}\right)
\end{align}
with, for any $(T,X)$, $\zeta\mapsto U_{(\epsilon)}(T,X;\zeta)$ periodic of period $1$ and, as $\epsilon\to0$,
\begin{align*}
U_{(\epsilon)}(T,X;\zeta)&=
U_{0}(T,X;\zeta)+\epsilon\,U_{1}(T,X;\zeta)+o(\epsilon)\,,\\
\varphi_{(\epsilon)}(T,X)&=
\varphi_{0}(T,X)+\epsilon\,\varphi_{1}(T,X)+o(\epsilon)\,.
\end{align*}
Plugging \eqref{eq:ansatz} into \eqref{eep-int} and identifying the leading-order terms yields
\begin{align*}
U_{0}(T,X;\zeta)&\,=\,\underline{U}^{(\delta,V)(T,X)}(\zeta)\,,&
\p_X\varphi(T,X)&=k^{(\delta,V)(T,X)}\,,&
\p_T\varphi(T,X)&=-V(T,X)\,k^{(\delta,V)(T,X)}\,,
\end{align*}
for some slow $(\delta,V)$. Note that this already implies
\begin{align}\label{eq:W1}
\p_T(k^{(\delta,V)})=\p_X(-V\,k^{(\delta,V)})\,.
\end{align}
To complete it to form a system for the evolution of $(\delta,V)$, one may insert \eqref{eq:ansatz} into \eqref{eq:cons}, identify the leading-order terms and average over one period of $\zeta$ so as to obtain
\begin{align}\label{eq:W2}
\p_T\left(\langle\,\mathcal{M}[\underline{U}^{(\delta,V)}(\cdot)]\,\rangle\right)
\,=\,\p_X\left(\langle\,\mathcal{S}[\underline{U}^{(\delta,V)}(\cdot)]\,\rangle\right)\,.
\end{align}
System \eqref{eq:W1}-\eqref{eq:W2} is the sought modulation system.

\subsection*{Spectral validation and small-amplitude limit}

The proof of Theorem~\ref{th4} contains the ingredients to obtain a spectral validation of  \eqref{eq:W1}-\eqref{eq:W2}. In particular it yields that if for some $(\delta,V)$ the linearization about $(\delta,V)$ of \eqref{eq:W1}-\eqref{eq:W2} possesses a non real characteristic velocity then there exists a positive $\varepsilon_0(\delta,V)$ such that for any $0<|\xi|<\varepsilon_0(\delta,V)$
\[
\sigma_{per}(L_{\xi}^{(\delta,V)})\cap B(0,\varepsilon_0(\delta,V))
\cap \{\,\lambda\,;\,\Re(\lambda)>0\,\}\,\neq\,\emptyset\,.
\]

Moreover the proof of Theorem~\ref{th4} also contains the elements to elucidate the foregoing criterion in the small-amplitude limit. The conclusion is that when $k_2^V<0$ this side-band instability does happen when $\delta$ is sufficiently small, whereas when $k_2^V>0$ this side-band instability cannot happen when $\delta$ is sufficiently small.

We point out that the small-amplitude analysis is simpler here than in \cite{Nonlinearity-BMR2,SMF-Audiard-Rodrigues} because we are actually analyzing the splitting of a double eigenvalue in a modulation system of two equations whereas for equations of Korteweg-de Vries type this double eigenvalue is burried in a modulation system of three equations, and for equations of Schr\"odinger or Euler-Korteweg type it is hidden in a modulation system of four equations.

Let us also observe that here, unlike what happens for systems considered in \cite{Nonlinearity-BMR2,SMF-Audiard-Rodrigues}, there is no limiting solitary wave and thus no large-period regime to analyze.

\section{Expansions of extension operators}\label{secA}

%%%%%%%%%%%%%%%%%%%%%%%%%%%%%%%%%%%%%%%%%%%%%%%%%%%%%%%%%%%%%%%%%%%%%%%%%%%%%
%%%%%%%%%%%%%%%%%%%%%%%% APPENDIX A.1%%%%%%%%%%%%%%%%%%%%%%%%%%%%%%%%%%%%%%%%%%%%%
%%%%%%%%%%%%%%%%%%%%%%%%%%%%%%%%%%%%%%%%%%%%%%%%%%%%%%%%%%%%%%%%%%%%%%%%%%%%%
\subsection{Proof of  Proposition~\ref{prop-extop}}\label{s:ext-xi}

Proposition~\ref{prop-extop} is a constant-coefficient periodic computation, thus it is convenient to introduce Fourier series to prove it. Since there is little risk of confusion, we use for Fourier series the same notation as for Fourier transforms, $\cF(g)=\widehat{g}$ being defined for functions of $L^1(\mR/\mZ)$ by
\[
(\cF g)(j)=\hat{g}_j:=\langle e^{i2\pi j\cdot},g\rangle
=\int_{0}^{1} e^{-i2\pi jx} g(x)\, \d x\,,\qquad j\in\mZ\,.
\]

Our goal is to prove for $\cU_\xi:=\cU^0(\xi)$,
\beq\label{exp-extop-F}
\begin{aligned}
  \cF\big(\cU_{\xi}g\big)(j)&=\hat{g}_j, \quad j\neq \pm 1, \\
  \cF\big(\cU_{\xi} g\big)(\pm 1)&= \hat{g}_j\mp
  \alpha\,\diag(1,-1)
 \hat{g}_{\pm 1}+\f{1}{2}\big(
\alpha^2 \hat{g}_{\pm 1}
  + \beta\,\diag (1,-1)\hat{g}_{\pm 1}\big)\xi^2+\cO(\xi^3)\,\|g\|_{L^2},
\end{aligned}
\eeq
where $\alpha$, $\beta$ are given by \eqref{defalp-beta}. The starting point is that from the Cauchy problem defining $\cU_\xi$ stems
\begin{align}\label{exp-extop-0}
    \cU_\xi=I + \xi\,[\Pi_0', \Pi_0]+\f{1}{2}\xi^2\left([\Pi_0',\Pi_0]^2+[\Pi_0'',\Pi_0]\right)+\cO(\xi^3)\,,
\end{align}
where $\Pi_{\xi}:=\Pi_{\xi}^0$.

Note that for any $\xi$, $\Pi_\xi$ is a projector with range spanned by $\varphi^{-1}_+(\xi,\cdot)$ and $\varphi^{1}_-(\xi,\cdot)$ (with notation from \eqref{defbasis-constant}), that are trigonometric monomials respectively in $e^{-i\,2\pi\cdot}$ and $e^{i\,2\pi\cdot}$. Therefore for any $\xi$,
\begin{align*}
\cF(\Pi_{\xi} g)(j)&=0\,,\quad\textrm{if }j\notin\{-1,1\}\,,\\
\Pi_{\xi} g&=0\,,\quad\textrm{if }\cF(\Pi_{\xi} g)(1)=0\textrm{ and }\cF(\Pi_{\xi} g)(-1)=0\,.
\end{align*}
Combined with \eqref{exp-extop-0}, this yields the first half of \eqref{exp-extop-F}.

For the remaining part, we use Cauchy Residue theorem to obtain
\begin{align*}
    \cF(\Pi_{\xi} g)(\pm 1)= \f{1}{\lambda_{\pm}^{\mp 1}-\lambda_{\mp}^{\mp 1}}
    \left(\lambda_{\pm}^{\mp 1}-\Sigma_0^L (i\,k^{(0)}(\pm\,2\pi+\xi))\right)^{\dagger} \hat{g}_{\pm 1},
\end{align*}
where, for any square matrix $A$, $A^\dagger$ denotes the transpose of the cofactor matrix of $A$ so that if $A$ is invertible $A^{-1}=A^\dagger/\det(A)$. To derive the former, we have used
\begin{align*}
\cF(L_{\xi}^0 g)(j)&\,=\,\Sigma_0^L (i\,k^{(0)}(2\pi\,j+\xi))\,\hat{g}_j\,,\\
\det(\lambda\,I-\Sigma_0^L (i\,k^{(0)}(2\pi\,j+\xi)))
&=(\lambda-\lambda_+^j(\xi))\,(\lambda-\lambda_-^j(\xi))\,.
\end{align*}
As a result,
\begin{align}\label{pixipm1}
     \cF(\Pi_{\xi} f)(\pm 1)=\cM_{\mp}^{\pm 1}(\xi) \hat{f}(\pm 1)
\end{align}
where, for any $(j,\xi)$,
\begin{align}\label{defcM}
    \cM_{\pm}^j(\xi):= \f{1}{2}\left(\begin{array}{cc}
       1 & \mp \f{1}{i\,\omega_j(\xi)} \\[5pt]
       \mp {i\,\omega_j(\xi)}  &  1
    \end{array}\right)\,.
\end{align}
Then, differentiating \eqref{pixipm1} and taking into account that $\omega_{-1}(\xi)=\omega_1(-\xi)$, one receives
\begin{align*}
\cF(\Pi_{\xi} f)(\pm 1)&=
     \f{i}{2}\left(\begin{array}{cc}
       0 & \mp \f{1}{\omega_{\pm1}(\pm\xi)} \\[5pt]
       \pm {\omega_{\pm1}(\pm\xi)}  & 0
    \end{array}\right) \hat{f}(\pm 1)\\
     \cF(\Pi_{\xi}' f)(\pm 1)&=
     \f{i}{2}\left(\begin{array}{cc}
       0 & \pm \f{\omega_{\pm1}'(\xi)}{(\omega_{\pm1}(\xi))^2} \\[5pt]
       \pm {\omega_{\pm1}'(\xi)}  & 0
    \end{array}\right) \hat{f}(\pm 1)
    =\f{i}{2}\left(\begin{array}{cc}
       0 & \f{\omega_1'(\pm\xi)}{(\omega_{1}(\pm\xi))^2} \\[5pt]
       {\omega_{1}'(\pm\xi)}  & 0
    \end{array}\right) \hat{f}(\pm 1)\\
         \cF(\Pi_{\xi}'' f)(\pm 1)&=
     \f{i}{2}\left(\begin{array}{cc}
       0 & \pm \f{\omega_{\pm1}''(\xi)\omega_{\pm1}(\xi)-2(\omega_{\pm1}'(\xi))^2}{(\omega_{\pm1}(\xi))^3} \\[5pt]
       \pm \omega_{\pm1}''(\xi)  & 0
    \end{array}\right) \hat{f}(\pm 1)\\
    &=\f{i}{2}\left(\begin{array}{cc}
       0 & \pm \f{\omega_{1}''(\pm\xi)\omega_{1}(\pm\xi)-2(\omega_{1}'(\pm\xi))^2}{(\omega_{1}(\pm\xi))^3} \\[5pt]
       \pm \omega_{1}''(\pm\xi)  & 0
    \end{array}\right) \hat{f}(\pm 1)\,.
\end{align*}
Since for any numbers $\alpha_1$, $\beta_1$, $\alpha_2$, $\beta_2$,
\[
\left[\begin{pmatrix}0&\alpha_1\\\beta_1&0\end{pmatrix},
\begin{pmatrix}0&\alpha_2\\\beta_2&0\end{pmatrix}\right]
\,=\,(\alpha_1\beta_2-\alpha_2\beta_1)\,\diag (1,-1)\,,
\]
one deduces that
\begin{align*}
    \cF([\Pi_0',\Pi_0] f)(\pm 1)&=\mp \alpha\, \diag (1,-1) \hat{f}(\pm 1)\,,\\
    \cF([\Pi_0'',\Pi_0] f)(\pm 1)&= \beta\, \diag (1,-1) \hat{f}(\pm 1)\,,
\end{align*}
from which the proof follows.

%%%%%%%%%%%%%%%%%%%%%%%%%%%%%%%%%%%%%%%%%%%%%%%%%%%%%%%%%%%%%%%%%%%%%%%%%%%%%
%%%%%%%%%%%%%%%%%%%%%%%% APPENDIX A.2%%%%%%%%%%%%%%%%%%%%%%%%%%%%%%%%%%%%%%%%%%%%%
%%%%%%%%%%%%%%%%%%%%%%%%%%%%%%%%%%%%%%%%%%%%%%%%%%%%%%%%%%%%%%%%%%%%%%%%%%%%%

\subsection{Proof of Proposition~\ref{prop:ext-delta}}\label{s:ext-delta}

For the sake of concision let us set $\Pi_\delta:=\Pi_{\xi_0}^\delta$ and $R^\lambda:=(\lambda\,I-L_{\xi_0}^0)$.

Then expanding from the Cauchy problem defining $\cU_{\xi_0}^\delta$,
\begin{align*}
\cU^{[1]}&=[\Pi'_0,\Pi_0]\,,&
\cU^{[2]}&=
\f{1}{2}\big(\big[\Pi_0'', \Pi_0\big]
+[\Pi_0', \Pi_0\big]^2 \big)\,,
\end{align*}
and
\begin{align*}
\cU^{[3]}&=
\f{1}{6}\bigg(\big[\Pi_0''',\Pi_0\big]
+\big[\Pi_0'',\Pi_0'\big]+\big[\Pi_0', \Pi_0\big]\big[\Pi_0'', \Pi_0\big]
+ 2\big[\Pi_0'', \Pi_0\big]\big[\Pi_0', \Pi_0\big]
+\big[\Pi_0', \Pi_0\big]^3 \bigg)\\
&=\f{1}{6}\big[\Pi_0''',\Pi_0\big]
+\frac13\,\big[\Pi_0', \Pi_0\big]\,\cU^{[2]}
+\f{1}{6}\bigg(\big[\Pi_0'',\Pi_0'\big]
+ 2\big[\Pi_0'', \Pi_0\big]\big[\Pi_0', \Pi_0\big]
+\big[\Pi_0', \Pi_0\big]^3\bigg)\,.
\end{align*}
As in Appendix~\ref{s:ext-xi}, we then observe that
\begin{align*}
 \Pi_0 (e^{2i\pi m \,(\cdot)}X)
=\begin{cases}
     e^{2i\pi m (\cdot)}\cM_{-}\,X\,,&\textrm{if }m=j\,,\\
     e^{2i\pi m (\cdot)}\cM_{+}\,X\,,&\textrm{if }m=j'\,,\\
        0\,, &\textrm{otherwhise}.
  \end{cases}
\end{align*}

To compute expansions of the projector, we rely on
 \begin{align*}
\p_\delta\big(\lambda I-L_{\xi_0}^{\delta}\big)^{-1}\big|_{\delta=0}
&=
%\big(\lambda I-L_{\xi_0}^{0}\big)^{-1}
R^{\lambda}\,L^{[1]}\,R^{\lambda}\,,\\
\p_\delta^2\big(\lambda I-L_{\xi_0}^{\delta}\big)^{-1}\big|_{\delta=0}
&=
2\bigg(R^{\lambda} L^{[1]} R^{\lambda} L^{[1]} R^{\lambda}
+R^{\lambda} L^{[2]} R^{\lambda}
\bigg),%\label{secdRxilambda}
\\
\p_\delta^3\big(\lambda I-L_{\xi_0}^{\delta}\big)^{-1}\big|_{\delta=0}
&= 6\bigg(R^{\lambda}L^{[1]}R^{\lambda} L^{[1]} R^{\lambda}L^{[1]}R^{\lambda}
+R^{\lambda} L^{[2]} R^{\lambda} L^{[1]} R^{\lambda}%\bigg)\\
%&+6\bigg(
+R^{\lambda} L^{[1]} R^{\lambda} L^{[2]} R^{\lambda}
+ R^{\lambda} L^{[3]} R^{\lambda}\bigg). %\label{thirdRxilambda}
\end{align*}
From the Cauchy residue theorem, we deduce
\begin{align*}
(\Pi_0')_{1}^{j'}&=G_{j'+1}(L^{[1]})_{1}^{j'}\cM_{+}\,,&
(\Pi_0')_{-1}^{j}&=G_{j-1}(L^{[1]})_{-1}^{j}\cM_{-}\,,&\\
(\Pi_0')_{-1}^{j'+1}&=\cM_{+}(L^{[1]})_{-1}^{j'+1} G_{j'+1}\,,&
(\Pi_0')_{1}^{j-1}&=\cM_{-}(L^{[1]})_{1}^{j-1} G_{j-1}\,.
\end{align*}
Inserting this in the above formula for $\cU^{[1]}$ proves the claim on $(\cU^{[1]})^j_{-1}$ and $(\cU^{[1]})^{j'}_1$.

Likewise, applying again the Cauchy residue theorem, we deduce
\begin{align*}
\frac12(\Pi_0'')_{2}^{j'}&=B_+\cM_{+}\,,&
\frac12(\Pi_0'')_{-2}^{j}&=B_-\cM_-\,,&\\
\frac12(\Pi_0'')_{-2}^{j'+2}&=\cM_+\,C_-\,,&
\frac12(\Pi_0'')_{2}^{j-2}&=\cM_-\,C_+\,.
\end{align*}
Inserting this in the above formula for $\cU^{[2]}$ proves the claim on $(\cU^{[2]})^j_{-2}$ and $(\cU^{[2]})^{j'}_2$.

To complete the proof of the proposition, we compute
\begin{align*}
\bigg(\big[\Pi_0'',\Pi_0'\big]\bigg)_{-3}^j&
=2\cM_{+} C_{-} G_{j-1} (L^{[1]})_{-1}^{j} \cM_{-}
-2\cM_{+} (L^{[1]})_{-1}^{j-2} G_{j-2} B_{-}\cM_{-}\,,\\
\bigg(\big[\Pi_0'',\Pi_0'\big]\bigg)_{3}^{j'}
&=2\cM_{-} C_{+} G_{j'+1} (L^{[1]})_{1}^{j'} \cM_{+}
-2 \cM_{-} (L^{[1]})_{1}^{j'+2} G_{j'+2} B_{+} \cM_{+}\,.
\end{align*}
and
\begin{align*}
\bigg(\big[\Pi_0'', \Pi_0\big]\big[\Pi_0', \Pi_0\big]\bigg)_{-3}^j
&=-2\cM_{+} C_{-} G_{j-1} (L^{[1]})_{-1}^{j}\cM_{-}\,,\\
\bigg(\big[\Pi_0'',\Pi_0\big]\big[\Pi_0', \Pi_0\big]\bigg)_{3}^{j'}
&=-2\cM_{-} C_{+} G_{j'+1} (L^{[1]})_{1}^{j} \cM_{+}\,.
\end{align*}
Combining these concludes the proof.

\section{Proof of Lemma~\ref{dl-c-s}}\label{s:computations}

To prove Lemma~\ref{dl-c-s}, our first task is to make the formula from Proposition~\ref{prop-rmc} even more explicit. To do so, we set
\begin{align*}
\kappa_0&:=2\pi k_0\,,&
s&:=j+\frac{\xi_0}{2\pi}-\frac32=j'+\frac{\xi_0}{2\pi}-\frac32\,,
\end{align*}
and note that, inserting \eqref{uud}-\eqref{ued1}-\eqref{ued2}-\eqref{ued3} in the definition of $L_\xi^\delta$ directly gives
\begin{align*}%\label{L0jmm}
    (L^{[0]})_0^{j-m}=
    \bigg( \begin{array}{cc}
     i \kappa_0 V\,(\f{3}{2}-m+s)  & -1 \\
      1+F''(1)\kappa_0^2\,(\f{3}{2}-m+s)^2 & i \kappa_0 V(\f{3}{2}-m+s)
    \end{array}
    \bigg),
\end{align*}
\begin{align*}%\label{L1m1}
    (L^{[1]})_{-1}^{j-m}=-\f{1}{2}
    \bigg( \begin{array}{cc}
     i \kappa_0 V( \f{3}{2}-m+s)  & 1 \\
      -F'''(1)\kappa_0^2(\f{3}{2}-m+s)(\f{1}{2}-m+s)& i \kappa_0 V(\f{1}{2}-m+s)
    \end{array}
    \bigg),
\end{align*}
%\begin{align}
%    (L^{[1]})_{1}^{j'+m}=-\f{1}{2}
%    \bigg( \begin{array}{cc}
%     i \kappa_0 V( -\f{3}{2}+m+s)  & 1 \\
%      -F'''(1)\kappa_0^2(-\f{3}{2}+m+s)(-\f{1}{2}+m+s)& i \kappa_0 V(-\f{1}{2}+m+s)
%    \end{array}
%    \bigg),
%\end{align}
\begin{align*}%\label{L2m2}
    (L^{[2]})_{-2}^{j-m}=\f{1}{2}
    \bigg( \begin{array}{cc}
     i \kappa_0\sigma_0( \f{3}{2}-m+s)  &  \f{h'(1)}{3 h(1)}  \\
      -\kappa_0^2\sigma_1(\f{3}{2}-m+s)(-\f{1}{2}-m+s) & i \kappa_0\sigma_0(-\f{1}{2}-m+s)
    \end{array}
    \bigg),
\end{align*}
%\begin{align}
%    (L^{[2]})_{2}^{j'+m}=\f{1}{2}
%    \bigg( \begin{array}{cc}
%     i \kappa_0\sigma_0( -\f{3}{2}+m+s)  &  \f{h'(1)}{3 h(1)}  \\
%      -\kappa_0^2\sigma_1(-\f{3}{2}+m+s)(\f{1}{2}+m+s) & i \kappa_0\sigma_0(\f{1}{2}+m+s)
%    \end{array}
%    \bigg),
%\end{align}
\begin{align*}%\label{L3j}
    (L^{[3]})_{-3}^{j}=-\f{1}{2}
    \bigg( \begin{array}{cc}
     i \kappa_0\sigma_3( \f{3}{2}+s)  &   \sigma_2 \\
      -\kappa_0^2\sigma_4(\f{3}{2}+s)(-\f{3}{2}+s)& i \kappa_0\sigma_3(-\f{3}{2}+s)
    \end{array}
    \bigg),
\end{align*}
where
\begin{align*}
%theta_2
\sigma_0&:=V\bigg(\f{1}{2}+\f{h'(1)}{3h(1)}\bigg)\,,&
%theta_1
\sigma_1&:=\f{F^{(3)}(1)h'(1)}{3 h(1)}-\f{F^{(4)}(1)}{4}\,,&
%theta_0
\sigma_2&:=\frac{3}{16}\left(\left(\frac{h'(1)}{h(1)}\right)^2-\frac14\frac{h''(1)}{h(1)}\right)\,,
\end{align*}
and
\begin{align*}
%theta_3
\sigma_3&:=V\bigg(\sigma_2+\f{1}{4}+\f{h'(1)}{3h(1)}\bigg)\,,&
%theta_4
\sigma_4&:=F^{(3)}(1)\sigma_2-\f{F^{(4)}(1)h'(1)}{6\,h(1)}+\f{F^{(5)}(1)}{24}\,.
\end{align*}
Setting
\begin{align*}
\chi_{m}&:=\lambda_0-i\kappa_0\,V\,\left(m+\frac{\xi_0}{2\pi}\right)=i\kappa_0 V(j-m)-i\omega_{j}(\xi_0)\,,\\
d_m&:=\f{1}{(\lambda_0-\lambda_{+}^{m}(\xi_0))(\lambda_0-\lambda_{-}^{m}(\xi_0))}=\f{1}{\chi_m^2+\omega_{m}^2(\xi_0)}\,,
\end{align*}
we also have
\[
G_{m}=
\begin{pmatrix}
0 & 0\\[2pt]
1  &    0
\end{pmatrix}
+d_m\,\begin{pmatrix}
1 \\[2pt]
-\chi_m
\end{pmatrix}
\begin{pmatrix}\chi_m& -1\end{pmatrix}\,.
\]

Next, using the foregoing computations, we make more explicit different parts of $\cN$, beginning with
\begin{align*}
&(L^{[1]})_{-1}^{j-2}G_{j-2} (L^{[2]})_{-2}^j\\
&\,=\,-\f{1}{4}
    \begin{pmatrix}
     1 \\
    i \kappa_0 V(-\f{3}{2}+s)
    \end{pmatrix}
    \begin{pmatrix}
     i \kappa_0\sigma_0( \f{3}{2}+s)  &  \f{h'(1)}{3 h(1)}
\end{pmatrix}%\\
%&\quad
-\f{d_{j-2}}{4}
    \begin{pmatrix}
     i \kappa_0 V(-\f{1}{2}+s)-\chi_{j-2}\\
      -\alpha_1(-\f{3}{2}+s)
    \end{pmatrix}
\begin{pmatrix}
     \alpha_2( \f{3}{2}+s)
     &-\alpha_3
\end{pmatrix}
\end{align*}
with
\begin{align*}
%alpha_4
\alpha_1&:=i\kappa_0 V\chi_{j-2}+F'''(1)\kappa_0^2\left(-\f{1}{2}+s\right)\,,\\
%alpha_6
\alpha_2&:=i\kappa_0\sigma_0\chi_{j-2}+\kappa_0^2\sigma_1\left(-\f{1}{2}+s\right)\,,\\
%alpha_5
\alpha_3&:=-\f{h'(1)}{3 h(1)}\chi_{j-2}+i \kappa_0\sigma_0\left(-\f{1}{2}+s\right)\,.
\end{align*}
Likewise
\begin{align*}
&(L^{[2]})_{-2}^{j-1}G_{j-1} (L^{[1]})_{-1}^j\\
&=-\f{1}{4}
\begin{pmatrix} \f{h'(1)}{3 h(1)}  \\
i \kappa_0\sigma_0(-\f{3}{2}+s)
    \end{pmatrix}
 \begin{pmatrix}
     i \kappa_0 V( \f{3}{2}+s)  & 1
      \end{pmatrix}
-\f{d_{j-1}}{4}
\begin{pmatrix}
\alpha_4\\
-\alpha_5(-\f{3}{2}+s)
\end{pmatrix}
\begin{pmatrix}
     \alpha_6(\f{3}{2}+s)
      &\chi_{j-1}-i \kappa_0 V(\f{1}{2}+s)
\end{pmatrix}
\end{align*}
with
\begin{align*}
%alpha_7
\alpha_4&:= i \kappa_0\sigma_0\left( \f{1}{2}+s\right)-\chi_{j-1}\f{h'(1)}{3 h(1)}\,,\\
%alpha_8
\alpha_5&:=\kappa_0^2\sigma_1\left(\f{1}{2}+s\right)+i \kappa_0\sigma_0\chi_{j-1}\,,\\
%alpha_2
\alpha_6&:=i \kappa_0 V\chi_{j-1}+F'''(1)\kappa_0^2\left(\f{1}{2}+s\right)\,.
\end{align*}
At last,
\begin{align*}
&(L^{[1]})_{-1}^{j-2}G_{j-2}(L^{[1]})_{-1}^{j-1}G_{j-1}(L^{[1]})_{-1}^{j}\\
&=-\f{1}{8}
\begin{pmatrix}
1 \\
i \kappa_0 V(-\f{3}{2}+s)
\end{pmatrix}
\begin{pmatrix}
     i \kappa_0 V( \f{3}{2}+s)  & 1
\end{pmatrix}\\
&\quad-\f{d_{j-1}}{8}\left(i \kappa_0 V\left(\f{1}{2}+s\right)-\chi_{j-1}\right)
\begin{pmatrix}
    1 \\
i \kappa_0 V(-\f{3}{2}+s)
\end{pmatrix}
\begin{pmatrix}
\alpha_6( \f{3}{2}+s)
&\chi_{j-1}-i\kappa_0 V(\f{1}{2}+s)
\end{pmatrix}\\
&\quad-\f{d_{j-2}}{8}\left(\chi_{j-2}-i \kappa_0 V\left(-\f{1}{2}+s\right)\right)
\begin{pmatrix}
     i \kappa_0 V(-\f{1}{2}+s)-\chi_{j-2}  \\
    -\alpha_1(-\f{3}{2}+s)
\end{pmatrix}
\begin{pmatrix}
     i \kappa_0 V( \f{3}{2}+s)  & 1
\end{pmatrix}\\
&\quad-\f{d_{j-1}\,d_{j-2}}{8}
\left(\alpha_1\left(\f{1}{2}+s\right)
-\chi_{j-1} \left(\chi_{j-2}-i \kappa_0 V\left(-\f{1}{2}+s\right) \right)\right)\\
&\qquad\times
\begin{pmatrix}
     i \kappa_0 V(-\f{1}{2}+s)-\chi_{j-2}\\
      -\alpha_1(-\f{3}{2}+s)
\end{pmatrix}
\begin{pmatrix}
\alpha_6( \f{3}{2}+s)& \chi_{j-1}-i \kappa_0 V(\f{1}{2}+s)
\end{pmatrix}\,.
\end{align*}

Inserting the foregoing in Proposition~\ref{prop-rmc} yields
\begin{align*}
\Gamma&=
\begin{pmatrix}-i\omega_{j'}(\xi_0)&1\end{pmatrix}\cN \begin{pmatrix}1\\i\omega_j(\xi_0)\end{pmatrix}%\\
%&
=\mathfrak{R}_1+ \mathfrak{R}_2+ \mathfrak{R}_3+ \mathfrak{R}_4
\end{align*}
with
\begin{align*}
\mathfrak{R}_1
&=\f{1}{2}\left(
\kappa_0\sigma_3\left(s\left(\omega_j-\omega_{j'}\right)-\f{3}{2}\left(\omega_{j'}+\omega_j\right)\right)
-\sigma_2\,\omega_j\omega_{j'}
+\kappa_0^2\sigma_4\left(s^2-\f{9}{4}\right)\right)\\
&\quad
+\f{1}{4}\left(-\omega_{j'}+\kappa_0 V\left(-\f{3}{2}+s\right)\right)
\left(\kappa_0\sigma_0\left(\f{3}{2}+s\right)
+\omega_j\f{h'(1)}{3 h(1)}\right)\\
&\quad+\f{1}{4}\left(-\omega_{j'}\f{h'(1)}{3 h(1)}
+\kappa_0\sigma_0\left(-\f{3}{2}+s\right)\right)
\left(\kappa_0 V\left(\f{3}{2}+s\right)+\omega_j\right)\\
&\quad+\f{1}{8}\left(-\omega_{j'}+\kappa_0 V\left(-\f{3}{2}+s\right)\right)
\left(\kappa_0 V\left( \f{3}{2}+s\right)+\omega_j\right)\,,
\end{align*}
\begin{align*}
\mathfrak{R}_2&=\f{d_{j-2}}{8}\left(
-\omega_{j'}
\left(i \kappa_0 V\left(-\f{1}{2}+s\right)-\chi_{j-2}\right)
+i\alpha_1\left(-\f{3}{2}+s\right)\right)\\
&\qquad\times
\left[-2\left(i\alpha_2\left(\f{3}{2}+s\right)
+\omega_j\alpha_3\right)
+\left(\chi_{j-2}-i \kappa_0 V\left(-\f{1}{2}+s\right)\right)
\left(\kappa_0 V\left( \f{3}{2}+s\right)+\omega_j\right)
\right]\,,
\end{align*}
\begin{align*}
\mathfrak{R}_3&=\f{d_{j-1}}{8}\left(-i\alpha_6\left(\f{3}{2}+s\right)
+\omega_j\left(\chi_{j-1}-i \kappa_0 V\left(\f{1}{2}+s\right)\right)\right)\\
&\qquad\times\left[2\left(-\omega_{j'}\alpha_4
+i\alpha_5\left(-\f{3}{2}+s\right)\right)
+\left(i \kappa_0 V\left(\f{1}{2}+s\right)-\chi_{j-1}\right)\,
\left(-\omega_{j'}+\kappa_0 V\left(-\f{3}{2}+s\right)\right)
\right]\,,
\end{align*}
and
\begin{align*}
\mathfrak{R}_4&:=\f{d_{j-1}\,d_{j-2}}{8}
\left(\alpha_1\left(\f{1}{2}+s\right)
-\chi_{j-1} \left(\chi_{j-2}-i \kappa_0 V\left(-\f{1}{2}+s\right) \right)\right)\\
&\qquad\times
\left(i\omega_{j'}\left(i \kappa_0 V\left(-\f{1}{2}+s\right)-\chi_{j-2}\right)
+\alpha_1\left(-\f{3}{2}+s\right)\right)\\
&\qquad\times
\left(\alpha_6\left(\f{3}{2}+s\right)
+i\omega_j\left(\chi_{j-1}-i \kappa_0 V\left(\f{1}{2}+s\right)\right)\right)\,,
\end{align*}
where we omit to specify that $\omega_j$ and $\omega_{j'}$ are evaluated at $\xi_0$.

Finally we specialize the discussion to power laws so as to expand in $V^{-1}$ in the limit $V\to\infty$. Actually we find it more convenient to carry out the expansion in $s^{-1}$ where $s$ is related to $V$ by
\[
s=\frac{\sqrt{5}}{2}\frac{V}{\sqrt{P'(1)}}
=\frac{\sqrt{5}}{2}\frac{V}{\sqrt{\gamma\,T}}\,.
\]
To begin with, we compute
\begin{align*}
F''(\rho)&=T\,\gamma\,\rho^{\gamma-2}\,,&
h(\rho;V)&=\frac{V^2}{\rho^3}-T\,\gamma\,\rho^{\gamma-2}\,,&
\end{align*}
and
\begin{align*}
\omega_j&=\sqrt{1+\gamma\,T\,\kappa_0^2\,\left(\frac32+s\right)^2}\,,&
\omega_{j'}&=\sqrt{1+\gamma\,T\,\kappa_0^2\,\left(-\frac32+s\right)^2}\,.
\end{align*}
Therefore
\begin{align*}
\kappa_0&=\frac{1}{\sqrt{h(1)}}=\frac{1}{V}\,\frac{1}{\sqrt{1-\frac54\,s^{-2}}}
\end{align*}
so that
\begin{align*}
\kappa_0\,V&=1+\frac58 s^{-2}+\cO(s^{-4})\,,&
\gamma\,T\,\kappa_0^2\,s^2&=\frac54\left(1+\frac{5}{4}s^{-2}\right)+\cO(s^{-4})\,,
\end{align*}
and
\begin{align*}
\omega_j&=\frac32\,\sqrt{1+\frac53 s^{-1}+\frac{35}{18}s^{-2}+\frac{25}{12} s^{-3}+\cO(s^{-4})}\\
&=\frac32\left(1+\frac12\left(\frac53 s^{-1}+\frac{35}{18}s^{-2}+\frac{25}{12} s^{-3}\right)
-\frac18\left(\frac53 s^{-1}+\frac{35}{18}s^{-2}\right)^2
+\frac{1}{16}\left(\frac53 s^{-1}\right)^3+\cO(s^{-4})\right)\\
&=\frac32+\frac54 s^{-1}+\frac{15}{16}s^{-2}
+\frac{25}{32}s^{-3}+\cO(s^{-4})\,.
\end{align*}
Thus
\begin{align*}
\omega_{j'}&=
\frac32-\frac54 s^{-1}+\frac{15}{16}s^{-2}
-\frac{25}{32}s^{-3}+\cO(s^{-4})\,,
\end{align*}
and
\begin{align*}
\chi_{j-1}&=i\,(\kappa_0\,V-\omega_j)
\,=\,-i\,\left(\frac12+\frac54 s^{-1}
+\frac{5}{16}s^{-2}
+\frac{25}{32}s^{-3}\right)+\cO(s^{-4})\,,\\
\chi_{j-2}&=i\,(2\kappa_0\,V-\omega_j)
\,=\,-i\,\left(-\frac12+\frac54 s^{-1}
-\frac{5}{16}s^{-2}
+\frac{25}{32}s^{-3}\right)+\cO(s^{-4})\,,
\end{align*}
\begin{align*}
d_{j-1}&=\frac{1}{\chi_{j-1}^2+\omega_{j-1}^2}
=\frac{1}{\chi_{j-1}^2+1+\gamma T\kappa_0^2\left(\frac12+s\right)^2}
=\frac12+\cO(s^{-4})\,,\\
d_{j-2}&=\frac{1}{\chi_{j-2}^2+\omega_{j-2}^2}
=\frac{1}{\chi_{j-2}^2+1+\gamma T\kappa_0^2\left(-\frac12+s\right)^2}
=\frac12+\cO(s^{-4})\,,
\end{align*}
\begin{align*}
i \kappa_0 V\left(\f{1}{2}+s\right)-\chi_{j-1}
&=i\,\left(s+1+\frac{15}{8}s^{-1}+\frac58 s^{-2}\right)+\cO(s^{-3})\,,\\
i \kappa_0 V\left(-\f{1}{2}+s\right)-\chi_{j-2}
&=i\,\left(s-1+\frac{15}{8}s^{-1}-\frac58 s^{-2}\right)+\cO(s^{-3})\,,
\end{align*}
\begin{align*}
%alpha_2
\alpha_6&=i \kappa_0 V\chi_{j-1}+(\gamma-2)T\gamma\kappa_0^2\left(\f{1}{2}+s\right)\\
&=\frac12+\frac54(\gamma-1)\,s^{-1}+\frac58(\gamma-1)s^{-2}
+\frac{25}{16}(\gamma-1)s^{-3}+\cO(s^{-4})\,,\\
%alpha_4
\alpha_1&=i\kappa_0 V\chi_{j-2}+(\gamma-2)T\gamma\kappa_0^2\left(-\f{1}{2}+s\right)\\
&=-\frac12+\frac54(\gamma-1)\,s^{-1}-\frac58(\gamma-1)s^{-2}
+\frac{25}{16}(\gamma-1)s^{-3}+\cO(s^{-4})\,.
\end{align*}

Likewise
\begin{align*}
\frac{h'(1)}{h(1)}
&=-\frac{3\,V^2+\gamma T\,(\gamma-2)}{V^2-\gamma T}
=-3-\frac54\frac{\gamma+1}{s^2-\frac54}\\
&=-3-\frac54(\gamma+1)\,s^{-2}+\cO(s^{-4})\,,\\
\frac{h''(1)}{h(1)}
&=\frac{12\,V^2-\gamma T\,(\gamma-2)(\gamma-3)}{V^2-\gamma T}
=12-\frac54\frac{(\gamma+1)(\gamma-6)}{s^2-\frac54}\\
&=12-\frac54(\gamma+1)(\gamma-6)\,s^{-2}+\cO(s^{-4})\,.
\end{align*}
Therefore
\begin{align*}
%theta_2
\kappa_0\,\sigma_0
&=\kappa_0V\bigg(\f{1}{2}+\f{h'(1)}{3h(1)}\bigg)
=-\frac12-\frac{5}{48}(4\gamma+7)s^{-2}+\cO(s^{-4})\,,\\
%theta_1
\kappa_0^2\,\sigma_1&=\kappa_0^2\left(\f{F^{(3)}(1)h'(1)}{3 h(1)}-\f{F^{(4)}(1)}{4}\right)
=-\frac{5}{16}(\gamma-2)(\gamma+1)s^{-2}+\cO(s^{-4})\,,\\
%theta_0
\sigma_2&=\frac{3}{16}\left(\left(\frac{h'(1)}{h(1)}\right)^2-\frac14\frac{h''(1)}{h(1)}\right)
=\frac98+\frac{15}{256}(\gamma+1)(\gamma+18)s^{-2}+\cO(s^{-4})\,,
\end{align*}
\begin{align*}
%theta_3
\kappa_0\,\sigma_3&=\kappa_0V\bigg(\sigma_2+\f{1}{4}+\f{h'(1)}{3h(1)}\bigg)
=\frac38+\frac{15}{64}\left(1+(\gamma+1)\left(\frac14\gamma+\frac{49}{18}\right)\right)s^{-2}
+\cO(s^{-4})\,,\\
%theta_4
\kappa_0^2\,\sigma_4&=
\kappa_0^2\left(F^{(3)}(1)\sigma_2-\f{F^{(4)}(1)h'(1)}{6\,h(1)}+\f{F^{(5)}(1)}{24}\right)%\\
%&
=\frac{5}{96}(\gamma-2)(\gamma^2+5\gamma+3)\,s^{-2}
+\cO(s^{-4})\,.
%+\frac{25}{3072}(\gamma-2)(49\gamma^2+147\gamma+90)\,s^{-4}+\cO(s^{-6})\,.
\end{align*}
Thus\begin{align*}
%alpha_7
\alpha_4&= i \kappa_0\sigma_0\left( \f{1}{2}+s\right)-\chi_{j-1}\f{h'(1)}{3 h(1)}
=-i\left(
\frac12\,s+\frac34+\frac{5}{48}(4\gamma+19)\,s^{-1}\right)
+\cO(s^{-2})\,,\\
%alpha_5
\alpha_3&=-\f{h'(1)}{3 h(1)}\chi_{j-2}+i \kappa_0\sigma_0\left(-\f{1}{2}+s\right)
=-i\left(
\frac12\,s-\frac34+\frac{5}{48}(4\gamma+19)\,s^{-1}\right)
+\cO(s^{-2})\,,
\end{align*}
and
\begin{align*}
%alpha_8
\alpha_5&=\kappa_0^2\sigma_1\left(\f{1}{2}+s\right)+i \kappa_0\sigma_0\chi_{j-1}
=-\frac14-\frac{5}{16}\gamma(\gamma-1)\,s^{-1}
-\frac{5}{96}(3\gamma^2+\gamma+4)\,s^{-2}+\cO(s^{-3})\,,\\
%alpha_6
\alpha_2&=i\kappa_0\sigma_0\chi_{j-2}+\kappa_0^2\sigma_1\left(-\f{1}{2}+s\right)
=\frac14-\frac{5}{16}\gamma(\gamma-1)\,s^{-1}
+\frac{5}{96}(3\gamma^2+\gamma+4)\,s^{-2}+\cO(s^{-3})\,.
\end{align*}

Then we compute that
\begin{align*}
\mathfrak{R}_1
&=-\frac18 s^2+\f{1}{64}\left(\frac53\gamma^3+5\gamma^2-25\gamma-\frac{103}{3}\right)+\cO(s^{-1})\,,
\end{align*}
%\begin{align*}
%\mathfrak{R}_2&=\f{1}{16}\left(
%-\frac32 s+\frac{11}{4}-5 s^{-1}+5s^{-2}
%-\frac12 s+\frac14(5\gamma-2)-\frac52(\gamma-1)s^{-1}+\frac52(\gamma-1)\,s^{-2}\right)\\
%&\qquad\times
%\Bigg[
%\frac12 s+\frac34-\frac{5}{8}\gamma(\gamma-1)+\frac{5}{24}\left(-3\gamma^2+5\gamma+2\right)s^{-1}
%-\frac32 s+1-\frac{15}{12}\left(\gamma+4\right)s^{-1}\\
%&\qquad
%+\,\left(s-1+\frac{15}{8}s^{-1}-\frac58 s^{-2}\right)
%\left(s+\frac32+\frac58 s^{-1}+\frac{15}{16}s^{-2}
%+\frac32+\frac54 s^{-1}+\frac{15}{16}s^{-2}\right)
%\Bigg]\\
%&=\f{1}{16}\left(
%-2 s+\frac14(5\gamma+9)-\frac52(\gamma+1)s^{-1}+\frac52(\gamma+1)\,s^{-2}\right)\\
%&\qquad\times
%\left[
%-s+\frac74-\frac{5}{8}\gamma(\gamma-1)-\frac{5}{24}\left(3\gamma^2+\gamma+22\right)s^{-1}
%+s^2+2s+\frac{3}{4}+5s^{-1}
%\right]\\
%&=\f{1}{16}\left(
%-2 s+\frac14(5\gamma+9)-\frac52(\gamma+1)s^{-1}+\frac52(\gamma+1)\,s^{-2}\right)\\
%&\qquad\times
%\left[
%s^2+s+\frac52-\frac{5}{8}\gamma(\gamma-1)-\frac{5}{24}\left(3\gamma^2+\gamma-2\right)s^{-1}
%\right]\\
%&=\f{1}{16}\left(-2s^3+\frac14(5\gamma+1)s^2+* s
%+\frac{5}{12}\left(3\gamma^2+\gamma-2\right)
%+\frac{5}{32}\left(4-\gamma(\gamma-1)\right)(5\gamma+9)\right)
%+\cO(s^{-1})\,,
%\end{align*}
\begin{align*}
\mathfrak{R}_2+\mathfrak{R}_3
%&=\f{1}{32}\left((5\gamma+1)s^2
%+\frac{5}{3}\left(3\gamma^2+\gamma-2\right)
%+\frac{5}{8}\left(4-\gamma(\gamma-1)\right)(5\gamma+9)\right)
%+\cO(s^{-1})\\
&=\f{1}{32}\left((5\gamma+1)s^2
-\frac{25}{8}\gamma^3+\frac52\gamma^2+\frac{475}{24}\gamma+\frac{115}{6}\right)
+\cO(s^{-1})\,,
\end{align*}
\begin{align*}
\mathfrak{R}_4%&=\f{1}{32}
%\left(-\frac12 s+\frac14(5\gamma-6)+\frac54(\gamma-1)s^{-2}
%+\frac12 s+\frac34+\frac52 s^{-2}
%\right)\\
%&\qquad\times
%\left(-\frac32 s+\frac{11}{4}-5 s^{-1}+5 s^{-2}
%-\frac12 s+\frac14(5\gamma-2)-\frac52(\gamma-1)s^{-1}+\frac52(\gamma-1)s^{-2}\right)\\
%&\qquad\times
%\left(\frac12 s+\frac14(5\gamma-2)+\frac52(\gamma-1)\,s^{-1}+\frac52(\gamma-1)\,s^{-2}
%+\frac32 s+\frac{11}{4}+5 s^{-1}+5 s^{-2}\right)\\
%&=\f{1}{32}
%\left(\frac14(5\gamma-3)+\frac54(\gamma+1)s^{-2}\right)\\
%&\qquad\times
%\left(-2s+\frac14(5\gamma+9)-\frac52(\gamma+1)s^{-1}%+\frac52(\gamma+1)s^{-2}
%\right)\\
%&\qquad\times
%\left(2s+\frac14(5\gamma+9)+\frac52(\gamma+1)s^{-1}%+\frac52(\gamma+1)s^{-2}
%\right)\\
%&=\f{1}{32}
%\left(\frac14(5\gamma-3)+\frac54(\gamma+1)s^{-2}\right)\\
%&\qquad\times
%\left(\left(\frac14(5\gamma+9)%+\frac52(\gamma+1)s^{-2}
%\right)^2
%-\left(2s+\frac52(\gamma+1)s^{-1}\right)^2\right)\\
%&=\f{1}{32}
%\left(\frac14(5\gamma-3)+\frac54(\gamma+1)s^{-2}\right)\\
%&\qquad\times
%\left(\frac{1}{16}(5\gamma+9)^2+\frac54(5\gamma+9)(\gamma+1)s^{-1}
%-4s^2-10(\gamma+1)\right)\\
%&=\f{1}{32}\left(-(5\gamma-3)s^2
%+\frac14(5\gamma-3)\left(\frac{1}{16}(5\gamma+9)^2-10(\gamma+1)\right)
%-5(\gamma+1)\right)+\cO(s^{-1})\\
&=\f{1}{32}
\left(-(5\gamma-3)s^2
+\frac{125}{64}\gamma^3-\frac{425}{64}\gamma^2
-\frac{505}{64}\gamma
-\frac{83}{64}\right)+\cO(s^{-1})\,.
\end{align*}
This yields Lemma~\ref{dl-c-s}.

\bibliographystyle{alphaabbr}
\bibliography{refep}
\end{document}